\documentclass[12pt,a4paper,psamsfonts]{amsart}
\usepackage{amssymb,amscd,amsxtra,calc}
\usepackage{cmmib57}

\newtheorem{thm}{Theorem}[section]
\newtheorem{lemma}[thm]{Lemma}
\newtheorem{proposition}[thm]{Proposition}
\newtheorem{corollary}[thm]{Corollary}
\newtheorem{claim}[thm]{Claim}

\theoremstyle{definition}
\newtheorem{definition}[thm]{Definition}
\newtheorem{remark}[thm]{Remark}
\newtheorem{notation}[thm]{Notation}
\newtheorem{ex}[thm]{Example}
\newtheorem{exmuk}{Example}

\newtheorem{example}{Example}

\newcommand{\pr}{\mathbb{P}}

\newcommand{\Z}{\mathbb{Z}}
\newcommand{\Q}{\mathbb{Q}}
\newcommand{\R}{\mathbb{R}}
\newcommand{\C}{\mathbb{C}}

\newcommand{\NE}{\operatorname{NE}}
\newcommand{\Nef}{\operatorname{Nef}}

\newcommand{\Eff}{\operatorname{Eff}}

\newcommand{\Bl}{\operatorname{Bl}}
\newcommand{\Exc}{\operatorname{Exc}}

\newcommand{\Spec}{\operatorname{Spec}}
\newcommand{\Pic}{\operatorname{Pic}}

\newcommand{\Ext}{\operatorname{Ext}}

\newcommand{\cont}{\operatorname{cont}}

\newcommand{\Mod}{\operatorname{mod}}

\newcommand{\sO}{\mathcal{O}}
\newcommand{\sN}{\mathcal{N}}
\newcommand{\sE}{\mathcal{E}}
\newcommand{\sQ}{\mathcal{Q}}
\newcommand{\sX}{\mathcal{X}}
\newcommand{\sL}{\mathcal{L}}
\newcommand{\sU}{\mathcal{U}}
\newcommand{\sH}{\mathcal{H}}

\newcommand{\cansub}{\subset_{\operatorname{can}}\,}

\title{Simple normal crossing Fano varieties and log Fano manifolds}
\author{Kento Fujita}

\begin{document}
\maketitle
\begin{abstract}{\noindent A projective log variety $(X, D)$ is called 
``a log Fano manifold" if $X$ is smooth and if $D$ is a reduced 
simple normal crossing divisor on $X$ with $-(K_X+D)$ ample. 
The $n$-dimensional log Fano manifolds $(X, D)$ with nonzero $D$ 
are classified in this article 
when the log Fano index $r$ of 
$(X, D)$ satisfies either $r\geq n/2$ with $\rho(X)\geq 2$ or $r\geq n-2$. 
This result is a partial generalization of the classification of 
logarithmic Fano threefolds by Maeda. }
\end{abstract}

\section{Introduction}\label{introsection}

As is well known, Fano varieties play an essential roll in various situations, especially 
in birational geometry. Many algebraic geometers have been studying 
Fano varieties, in both smooth and singular cases. 
During the past 30 years, researches of smooth Fano varieties, 
so called \emph{Fano manifolds}, have been advanced very much 
by using the theory of extremal rays; 
Fano manifolds can be classified up to dimension three. 
We have also known that the anticanonical degree of $n$-dimensional Fano manifolds 
are bounded for an arbitrary $n$. 
However, very little is known for higher dimensional case. 
Nowadays, many people use the \emph{Fano index} of a Fano manifold, 
the largest positive integer $r$ 
such that the anticanonical divisor is $r$ times a Cartier divisor, 
or the \emph{Fano pseudoindex} of a Fano manifold, the minimum 
of the intersection numbers of the anticanonical divisor with irreducible rational curves, 
in order to classify higher dimensional Fano manifolds. 
For example, the $n$-dimensional Fano manifolds with the Fano indices $r\geq n-2$ 
have been classified (\cite{KO,fujitabook,isk77,MoMu,mukai,wis,wisn90,wisn91}). 
Another example is the \emph{Mukai conjecture} \cite[Conjecture 4]{mukaiconj} 
(resp.\ \emph{generalized Mukai conjecture}), 
that the product of the Picard number and the Fano index 
(resp.\ the Fano pseudoindex) minus one 
is less than or equal to the dimension for any Fano manifold. 
The generalized Mukai conjecture is one of the most significant topics 
in the classification theory of Fano manifolds and is still open even now 
excepts for the case $n\leq 5$ or $\rho(X)\leq 3$ (see \cite{Occ,NO}). 
One of the most significant results related to the Mukai conjecture is due to Wi\'sniewski 
\cite{wisn90,wisn91}; he has classified $n$-dimensional Fano manifolds with 
the Fano index $r>n/2$ and the Picard number $\rho(X)\geq 2$. 

On the other hand, it is difficult to classify singular irreducible ($\Q$-)Fano varieties 
since we cannot use deformation theory successfully. In this article, 
we consider \emph{simple normal crossing Fano varieties}, that is, 
projective simple normal crossing varieties 
whose dualizing sheaves are dual of ample invertible sheaves. 
The concept of simple normal crossing Fano varieties is one of the most natural 
generalization of that of Fano manifolds. It also seems to be more tractable than 
that of singular irreducible ($\Q$-)Fano varieties. 
We also consider the \emph{snc Fano indices} 
(resp.\ the \emph{snc Fano pseudoindices})
of simple normal crossing Fano varieties; the largest positive integer $r$ 
such that the dualizing sheaf is $(-r)$-th power of 
some ample invertible sheaf (resp.\ the minimum of the intersection number of the 
dual of the dualizing sheaf and a rational curve). 
It is expected that 
simple normal crossing Fano varieties with large snc Fano indices 
have various applications 
as smooth Fano varieties do (for example, see \cite{Kol11}). 

It is natural to consider their irreducible components with the conductor divisors 
in order to investigate simple normal crossing Fano varieties. 
The component with the conductor has been considered by Maeda in his studies of 
a logarithmic Fano variety \cite{Maeda}. He has classified 
logarithmic Fano varieties of dimension at most three. 
A logarithmic Fano variety is called a \emph{log Fano manifold} in this article, 
which is a pair $(X, D)$ consisting of a smooth projective variety $X$ 
and a reduced simple normal crossing divisor $D$ on $X$ such that $-(K_X+D)$ is ample. 
These manifolds share similar properties with Fano manifolds but 
there are some differences between them. 
For example, as Maeda pointed out, 
the degree $(-(K_X+D)^{\cdot n})$ cannot be bounded 
for log Fano manifolds $(X, D)$ of fixed dimension $n\geq 3$.  

We also introduce the concept of the \emph{log Fano index} 
(resp.\ the \emph{log Fano pseudoindex}) of a log Fano manifold, 
the largest positive integer $r$ 
dividing $-(K_X+D)$, i.e., $-(K_X+D)\sim rL$ for a Cartier divisor $L$ 
(resp.\ the minimum of the intersection number of $-(K_X+D)$ and a rational curve). 
Our main new treatment is to consider log Fano manifolds 
with the log Fano indices (or the log Fano pseudoindices). 

The main idea to investigate $n$-dimensional log Fano manifolds $(X, D)$ 
with the log Fano index $r$ (resp.\ the log Fano pseudoindex $\iota$) 
and $D\neq 0$ is the following. 
First, to see the contraction morphism associated to an extremal ray intersecting 
the log divisor positively. We can classify del Pezzo type log Fano manifolds 
using only this strategy in Section \ref{delpezzo_section}. 
Second, one can show that $D$ is an $(n-1)$-dimensional simple normal crossing 
Fano variety such that the snc Fano index is divisible by $r$
(resp.\ the snc Fano pseudoindex is at least $\iota$). 
Hence we can use inductive arguments; 
the larger the snc Fano index or the snc Fano pseudoindex be, 
the simpler the structure of simple normal crossing 
Fano varieties be, as for Fano manifolds. 
As a consequence, we can analyze the structure of $X$ by viewing the 
restriction of contraction morphisms to $D$. 

Now we quickly organize this article. 
Section \ref{junbisection} is a preliminary section. 
We introduce the notion of simple normal crossing varieties and log manifolds 
in Section \ref{snc_section}. 
We also investigate the condition when log manifolds 
(and invertible sheaves on them) are glued as a simple normal crossing variety 
(and an invertible sheaf on it) in Theorem \ref{glue} and Proposition \ref{picglue}. 
We define simple normal crossing Fano varieties and log Fano manifolds 
in Section \ref{sncfano_section} and quickly give some properties 
in Section \ref{firstprop_section}. 
We also show some basic results of bundle structures (Section \ref{bdle_section}), 
extremal contractions (Section \ref{extcont_section}) and the special projective bundles, 
so called the Hirzebruch--Kleinschmidt varieties (Section \ref{property_scroll}). 

In Section \ref{ex_section}, we give various examples of log Fano manifolds with 
large log Fano indices, which occur in the theorems 
in Section \ref{thm_section}. 

In Section \ref{thm_section}, we state the main results of this article. 

In Section \ref{delpezzo_section}, we treat an $n$-dimensional 
log Fano manifold $(X, D)$ with $D\neq 0$ such that the log Fano pseudoindex 
is equal to $n-1$ 
(Proposition \ref{dP}), whose proof is easy but is optimal to understand how to classify 
log Fano manifolds with index data. 

The main purpose of this article, which we discuss in Section \ref{mukaiconjsection}, 
is to classify $n$-dimensional log Fano 
manifolds of the log Fano indices $r$ with nonzero boundaries such that 
$n\leq 2r$ and $\rho(X)\geq 2$ 
(Theorem \ref{mukai0} and Theorem \ref{mukai1}, see also Table \ref{maintable}), 
which is a log version of the treatment of the Mukai conjecture 
by Wi\'sniewski \cite{wisn90,wisn91}. We prove Theorem \ref{mukai1} in 
Section \ref{prf_section}. 
We note that Wi\'sniewski argued the case 
$r>n/2$ and we treat the case $r\geq n/2$ and nonzero boundaries. 
We also note that we do not treat Maeda's case $n=3$ and $r=1$; 
some of the techniques of the proof are similar to Maeda's one but 
we investigate completely different objects to Maeda's one. 

\begin{thm}[= Theorem \ref{mukai0}]\label{main0_intro}
If $(X,D)$ is an $n$-dimensional log Fano manifold 
with the log Fano pseudoindex $\iota>n/2$, $D\neq 0$ and $\rho(X)\geq 2$, 
then $n=2\iota-1$ and 
\[
(X, D)\simeq(\pr[\pr^{\iota-1};0^\iota, m], \,\,
\pr^{\iota-1}\times\pr^{\iota-1})
\]
with $m\geq 0$, 
where the embedding $D\subset X$ is obtained by the canonical projection 
$\pr[\pr^{\iota-1};0^\iota]\cansub\pr[\pr^{\iota-1};0^\iota, m]$. 
$($This is exactly the case 
in Example \ref{2rminusone} in Section \ref{ex_section}.$)$ 
\end{thm}

\begin{thm}[= Main Theorem \ref{mukai1}]\label{main1_intro}
Let $(X, D)$ be a $2r$-dimensional log Fano manifold with the log Fano index $r\geq 2$,  $D\neq 0$ and $\rho(X)\geq 2$. 
Then $(X, D)$ is exactly in the Examples 
\ref{burouappu}, \ref{pP}, \ref{kayaku}, \ref{fanoQ}, \ref{rthree}, 
\ref{rtwo}, \ref{Tp}, \ref{Pp}, \ref{zerozeroone}, \ref{zeroonebig}, 
\ref{zerozerobig} 
$($See Table \ref{maintable} for the list of $(X, D)$$).$
\end{thm}

As a consequence of Theorems \ref{mukai0}, \ref{mukai1}, combining Maeda's result, 
we have classified $n$-dimensional log Fano manifolds 
with the log Fano indices $r\geq n-2$ and nonzero boundaries, 
which we discuss in Section \ref{mukaisection} (Corollary \ref{coindex3}).

\smallskip

\noindent\textbf{Acknowledgements.}
The author is grateful to Professor J\'anos Koll\'ar for showing him the earlier version 
of \cite{Kol11} and for making a suggestion to classify snc Fano varieties. 
He also would like to expresses his gratitude to Professors Shigefumi Mori 
and Noboru Nakayama 
for their warm encouragements and for making various suggestions. 
He thanks Professor Eiichi Sato and Doctor Kazunori Yasutake for suggesting him to 
replace log Fano index by log Fano pseudoindex in Theorem \ref{mukai0} during the 
participation of Algebraic Geometry Seminar in Kyushu University. 
The author is partially supported by JSPS Fellowships for Young Scientists.

\medskip

\noindent\textbf{Notation and terminology.}
We always work in the category of algebraic 
(separated and finite type) schemes over a fixed algebraically closed field $\Bbbk$ 
of characteristic zero. 
A \emph{variety} means a connected and reduced algebraic scheme. 
For the theory of extremal contraction, we refer the readers to \cite{KoMo}. 
For a complete variety $X$, the Picard number of $X$ is denoted by  $\rho(X)$. 
For a smooth projective variety $X$, we define $\Eff(X)$ (resp.\ $\Nef(X)$) 
to be the effective (resp.\ nef) cone which is defined as the cone in $N^1(X)$ 
spanned by classes of effective (resp.\ nef) divisors on $X$. 
For a smooth projective variety $X$ and a $K_X$-negative extremal ray $R\subset\overline{\NE}(X)$,
\[
l(R):=\min\{(-K_X\cdot C)\mid C\text{ is a rational curve with } [C]\in R\}
\]
is called the \emph{length} $l(R)$ of $R$. 
A rational curve $C\subset X$ with 
$[C]\in R$ and $(-K_X\cdot C)=l(R)$ is called \emph{a minimal rational curve of $R$}.

For a morphism of algebraic schemes $f\colon X\rightarrow Y$, we define the 
\emph{exceptional locus} $\Exc(f)$ \emph{of $f$} by 
\[
\Exc(f):=\{x\in X\mid f \text{ is not isomorphism around }x\}.
\]

For a complete variety $X$, an invertible sheaf $\sL$ 
on $X$ and for a nonnegative integer $i$, 
let us denote the dimension of the $\Bbbk$-vector space $H^i(X,\sL)$ by $h^i(X,\sL)$. 
We also define $h^i(X,L)$ as $h^i(X,\sO_X(L))$
for a Cartier divisor $L$ on $X$. 

For algebraic schemes (or coherent sheaves on a fixed algebraic scheme) 
$X_1,\dots,X_m$, the projection is denoted by 
$p_{i_1,\dots,i_k}\colon\prod_{i=1}^mX_i\rightarrow\prod_{j=1}^kX_{i_j}$ for 
any $1\leq i_1<\cdots<i_k\leq m$. 

For an algebraic scheme $X$ and a locally free sheaf of finite rank $\sE$ on $X$, 
let $\pr_X(\sE)$ be the projectivization of $\sE$ in the sense of Grothendieck 
and $\sO_\pr(1)$ be the tautological invertible sheaf. We usually denote 
the projection by $p\colon\pr_X(\sE)\rightarrow X$. 
For locally free sheaves $\sE_1,\dots,\sE_m$ of finite rank on $X$ 
and $1\leq i_1<\cdots<i_k\leq m$, 
we sometimes 
denote the embedding obtained by the natural projection 
$p_{i_1,\dots,i_k}\colon\bigoplus_{i=1}^m\sE_i\rightarrow\bigoplus_{j=1}^k\sE_{i_j}$ by 
\[
\pr_X\Biggl(\bigoplus_{j=1}^k\sE_{i_j}\Biggr)\quad\cansub\quad
\pr_X\Biggl(\bigoplus_{i=1}^m\sE_i\Biggr)
\]
and we call that this embedding is \emph{obtained by the canonical projection}.

The symbol $\Q^n$ (resp.\ $\sQ^n$) means a smooth (resp.\ possibly non-smooth) 
hyperquadric in $\pr^{n+1}$ for $n\geq 2$. 
We write $\sO_{\Q^n}(1)$ (resp.\ $\sO_{\sQ^n}(1)$) as the invertible sheaf 
which is the restriction of $\sO_{\pr^{n+1}}(1)$ under the natural embedding. 
We sometimes write $\sO(m)$ instead of $\sO_{\Q^n}(m)$ 
(or $\sO_{\sQ^n}(m)$, $\sO_{\pr^n}(m)$) for simplicity. 

For an irreducible projective variety $V$ with $\Pic(V)=\Z$,  
the ample generator $\sO_V(1)$ of $\Pic(V)$, a nonnegative integer $t$ 
and integers $a_0,\dots,a_t$, we denote the projective space bundle 
\[
\pr_V(\sO_V(a_0)\oplus\dots\oplus\sO_V(a_t))\quad\text{by}\quad\pr[V; a_0,\dots,a_t]
\]
for simplicity. 
(We often denote 
\[
\pr[V; \underbrace{b_0,\dots,b_0}_{n_0 \text{ times}},\dots,
\underbrace{b_u,\dots,b_u}_{n_u \text{ times}}]
\quad\text{by}\quad\pr[V; b_0^{n_0},\dots,b_u^{n_u}]
\]
for any integers $b_0,\dots,b_u$ and positive integers $n_0,\dots,n_u$.)
We also denoted by $\sO(m; n)$ the invertible sheaf 
\[
p^*\sO_V(m)\otimes\sO_\pr(n)\quad\text{on}\quad\pr[V; a_0,\dots,a_t]
\]
for any integers $m$ and $n$, 
where $p\colon \pr[V; a_0,\dots,a_t]\rightarrow V$ is the projection and 
$\sO_\pr(1)$ is the tautological invertible sheaf with respect to $p$. 
For any $0\leq i_1<\cdots<i_k\leq t$, we denote the embedding 
\[
\pr_V(\sO_V(a_{i_1})\oplus\dots\oplus\sO_V(a_{i_k}))
\cansub\pr_V(\sO_V(a_0)\oplus\dots\oplus\sO_V(a_t))
\]
obtained by the canonical projection 
by $\pr[V; a_{i_1},\dots,a_{i_k}]\cansub\pr[V; a_0,\dots,a_t]$, and we also call that 
this embedding is \emph{obtained by the canonical projection}.

%\medskip

\section{Preliminaries}\label{junbisection}

%\smallskip

\subsection{Snc varieties and log manifolds}\label{snc_section}

First, we define simple normal crossing varieties and log manifolds.

\begin{definition}[normal crossing singularities]\label{ncdfn}
Let $X$ be a variety and $x\in X$ be a closed point. 
We say that $X$ has \emph{normal crossing singularity} at $x$ if 
$X$ is equi-dimensional of dimension $n$ around $x$ and if the completion of the 
local ring $\sO_{X,x}$ is isomorphic to 
\[
\Bbbk[[x_1,\dots,x_{n+1}]]/(x_1\cdots x_k)
\]
for some $1\leq k\leq n+1$. 
\end{definition}

\begin{definition}[snc varieties and log manifolds]\label{sncdfn}
\begin{enumerate}
\renewcommand{\theenumi}{\arabic{enumi}}
\renewcommand{\labelenumi}{(\theenumi)}
\item
A \emph{simple normal crossing variety} (\emph{snc} variety, for short) 
is a variety $\sX$ 
having normal crossing singularities at any closed points $x\in\sX$ 
and each irreducible component of $\sX$ is a smooth variety. 
\item
A \emph{log manifold} is a pair $(X,D)$ such that $X$ is a smooth variety 
and $D$ is a simple normal crossing divisor on $X$, that is, 
$D$ is reduced and each connected component of $D$ 
is an snc variety. 
\end{enumerate}
\end{definition}

\begin{remark}\label{nrmlrmk}
Let $\sX$ be an snc variety and $\nu\colon\overline{X}\rightarrow\sX$ be 
the normalization morphism of $\sX$. 
Then $\overline{X}$ is exactly the disjoint union 
of all the irreducible components of $\sX$ and $\nu$ is induced by natural embeddings. 
\end{remark}

\begin{definition}[conductor divisor]\label{conddfn}
Let $\sX$ be an snc variety with the irreducible decomposition 
$\sX=\bigcup_{1\leq i\leq m}X_i$. For any distinct $1\leq i$, $j\leq m$, 
the intersection $X_i\cap X_j$ can be seen as a smooth divisor $D_{ij}$ in $X_i$. 
We define 
\[
D_i:=\sum_{j\neq i}D_{ij}
\]
and call it the \emph{conductor divisor} in $X_i$ (with respect to $\sX$). 
We often write that 
$(X_i, D_i)\subset\sX$ is an irreducible component 
for the sake of simplicity. 
We also write $\sX=\bigcup_{1\leq i\leq m}(X_i, D_i)$ 
for emphasizing the conductor divisors. 
\end{definition}

\begin{remark}\label{adjunctionrmk}
If $\sX$ is an snc variety, then $\sX$ has 
an invertible dualizing sheaf $\omega_{\sX}$ since $\sX$ has 
only Gorenstein singularities. 
Furthermore, if $(X,D)\subset\sX$ is an irreducible component 
with the conductor divisor, then $(X, D)$ is a log manifold and 
\[
\omega_{\sX}|_X\simeq\sO_X(K_X+D)
\]
by the adjunction formula, where $K_X$ denotes the canonical divisor of $X$. 
\end{remark}

\begin{definition}[{strata and lc centers \cite{fujino}}]\label{lccenter}
\begin{enumerate}
\renewcommand{\theenumi}{\arabic{enumi}}
\renewcommand{\labelenumi}{(\theenumi)}
\item
Let $\sX$ be an snc variety with the irreducible decomposition 
$\sX=\bigcup_{1\leq i\leq m}X_i$.  
\begin{itemize}
\item
A \emph{stratum} of $\sX$ is an irreducible component of $\bigcap_{i\in I}X_i$ 
with the reduced scheme structure for a subset $I\subset\{1,\dots,m\}$. 
\item
A \emph{minimal stratum} of $\sX$ is a stratum of $\sX$ which is a minimal 
in the set of strata of $\sX$ under the partial order of the inclusion. 
\end{itemize}
\item
Let $(X, D)$ be a log manifold. 
\begin{itemize}
\item
An \emph{lc center} of $(X, D)$ is a 
stratum of some connected component of $D$. 
\item
A \emph{minimal lc center} of $(X, D)$ is an lc center of $(X, D)$ 
which is a minimal 
in the set of lc centers of $(X, D)$ under the partial order of the inclusion. 
\end{itemize}
\end{enumerate}
\end{definition}

Now, we give the theorem that 
decompressing log manifolds to snc varieties. 
Recall that an algebraic scheme has the \emph{Chevalley-Kleiman property} 
if every finite subscheme is contained in an open affine subscheme 
(see \cite[Definition 47]{quot}). For example, any projective variety has the 
Chevalley-Kleiman property. 

\begin{thm}\label{glue}
Let $(X_1, D_1),\dots,(X_m, D_m)$ be $n$-dimensional log manifolds such that 
$X_i$ has the Chevalley-Kleiman propeety for any $1\leq i\leq m$. 
Assume that $D_i$ is decomposed into 
\[
D_i=\sum_{j\neq i,\,\, 1\leq j\leq m}D_{ij}
\]
such that $D_{ij}$ are smooth $($possibly, non connected or empty$)$ divisor 
for any $1\leq i\leq m$, and 
there exists isomorphisms 
\[
\phi_{ij}\colon D_{ij}\rightarrow D_{ji}
\]
for all distinct $1\leq i, j\leq m$. We assume furthermore that 
\begin{enumerate}
\renewcommand{\theenumi}{\arabic{enumi}}
\renewcommand{\labelenumi}{\rm{(\theenumi)}}
\item\label{glue1}
$\phi_{ji}=\phi_{ij}^{-1}$ for any distinct $i$, $j$, 
\item\label{glue2}
$\phi_{ij}(D_{ij}\cap D_{ik})=D_{ji}\cap D_{jk}$ and 
$\phi_{jk}|_{D_{ji}\cap D_{jk}}\circ\phi_{ij}|_{D_{ij}\cap D_{ik}}=\phi_{ik}|_{D_{ij}\cap D_{ik}}$
for any distinct $i$, $j$, $k$.
\end{enumerate}
Then, there exists an algebraic scheme $\sX$ which has the Chevalley-Kleiman 
property such that any connected component 
of $\sX$ is an $n$-dimensional snc variety and the normalization of $\sX$ can be written as 
\[
\nu\colon X_1\sqcup\dots\sqcup X_m\rightarrow\sX,
\]
and $\sX$ also satisfies the following: 
\begin{itemize}
\item
The conductor divisor in $X_i$ (with respect to $\sX$) 
is exactly $D_i$ for any $1\leq i\leq m$. 
\item
$\nu(x_i)=\nu(x_j)$ if and only if $x_i\in D_{ij}$, $x_j\in D_{ji}$ and $\phi_{ij}(x_i)=x_j$ 
for any $1\leq i<j\leq m$, $x_i\in X_i$ and $x_j\in X_j$.  
\end{itemize}
\end{thm}

\begin{proof}[Sketch of the proof]
Let $X:=X_1\sqcup\dots\sqcup X_m$ and 
$D:=D_1\sqcup\dots\sqcup D_m\subset X$, a divisor on $X$. 
We note that $\{\phi_{ij}\}_{i\neq j}$ defines an involution on $\bigsqcup_{i\neq j}D_{ij}$, 
which is the normalization of $D$, by the condition \eqref{glue1}. 
It is easy to show that this involution generates a finite equivalence relation 
in the sense of Koll\'ar \cite[(26), \S 3]{kollarbook} by the condition \eqref{glue2}. 
Therefore, by \cite[Theorem 23, \S 3]{kollarbook} and \cite[Corollary 48]{quot} 
(see also \cite{fer}), we have a finite surjective morphism 
$\nu\colon X\rightarrow\sX$ 
onto a seminormal variety which has the Chevalley-Kleiman property 
such that the following property holds. 
For $x_i\in X_i$ and $x_j\in X_j$, $\nu(x_i)=\nu(x_j)$ holds if and only if 
either of the following holds; $i=j$ and $x_i=x_j$, or $i\neq j$, $x_i\in D_{ij}, x_j\in D_{ji}$ and 
$\phi_{ij}(x_i)=x_j$. 

Pick an arbitrary small affine open subscheme $\sU=\Spec A\subset\sX$. 
Since $\nu$ is a finite morphism, 
$U_i:=\nu^{-1}(\sU)\cap X_i$ is affine for any $1\leq i\leq m$.  
Let $U_i=\Spec A_i$. 
Then we have a finite extension of $\Bbbk$-algebra 
$A\hookrightarrow\prod_{i=1}^m A_i$ obtained by $\nu$. 
After replacing $\sU$ by a small one, $D_{ij}\cap U_i\subset U_i$ 
can be defined by a single element $f_{ij}\in A_i$ for any $i$, $j$. 
Let $\psi_{ij}\colon A_i/(f_{ij})\rightarrow A_j/(f_{ji})$ 
be the isomorphism of $\Bbbk$-algebra 
associated to the isomorphism 
$\phi_{ji}|_{D_{ji}\cap U_j}\colon D_{ji}\cap U_j\rightarrow D_{ij}\cap U_i$. 
Then it is easy to show the equality
\[
A=\Biggl\{(a_i)_{i}\in\prod_{i=1}^m A_i\ \Bigg|\ \psi_{ij}(a_i\Mod\, (f_{ij}))=a_j\Mod\, (f_{ji})
\quad\text{for all }i\neq j\Biggr\},
\]
the restriction $\nu|_{X_i}$ is a closed embedding for any $1\leq i\leq m$ 
and $\Spec A$ has normal crossing singularities at all closed points. 
\end{proof}

Next, we consider the descent of invertible sheaves. 

\begin{proposition}\label{picglue}
Let $\sX$ be an $n$-dimensional snc variety 
with the irreducible decomposition $\sX=\bigcup_{i=1}^mX_i$ 
which has a unique minimal stratum. 
We also let $X_{ij}:=X_i\cap X_j$ $($scheme theoretic intersection$)$ for any 
$1\leq i<j\leq m$. 
Then we have an exact sequence 
\[
0\rightarrow\Pic(\sX)\xrightarrow{\eta}\bigoplus_{i=1}^m\Pic(X_i)
\xrightarrow{\mu}\bigoplus_{1\leq i<j\leq m}\Pic(X_{ij}),
\]
where $\eta$ is the restriction homomorphism and 
\[
\mu\Bigl((\sH_i)_i\Bigr):=(\sH_i|_{X_{ij}}\otimes\sH_j^\vee|_{X_{ij}})_{i<j}.
\]
\end{proposition}

\begin{proof}
Let $\sX_i:=\bigcup_{j=1}^iX_j\subset\sX$ for any $1\leq i\leq m$. 
Then it is easy to show that both $\sX_i$ and $\sX_i\cap X_{i+1}$ are 
snc varieties and have a unique minimal stratum. 
Since units of structure sheaves form an exact sequence 
\[
1\rightarrow\sO_{\sX_{i+1}}^*\rightarrow\sO_{X_{i+1}}^*\times\sO_{\sX_i}^*
\rightarrow\sO_{\sX_i\cap X_{i+1}}^*\rightarrow 1
\]
 of sheaves of Abelian groups, 
which induces a long exact sequence 
\begin{eqnarray*}
1 & \rightarrow & \Bbbk^*\rightarrow\Bbbk^*\times\Bbbk^*
\xrightarrow{\upsilon}\Bbbk^* \\
& \rightarrow & \Pic(\sX_{i+1})\xrightarrow{\lambda}\Pic(X_{i+1})\oplus\Pic(\sX_i)
\rightarrow\Pic(\sX_i\cap X_{i+1}).
\end{eqnarray*}
Here $\sX_i$, $\sX_{i+1}$ and $X_{i+1}$ are all connected. 
The map $\lambda$ above is injective since $\upsilon$ is surjective. 
In particular, $\eta$ is injective. 

It is obvious that $\mu\circ\eta=0$. Assume that  
$(\sH_i)_i\in\bigoplus_{i=1}^m\Pic(X_i)$ satisfies $\mu\bigl((\sH_i)_i\bigr)=0$. 
We will construct an invertible sheaf $\sL_i\in\Pic(\sX_i)$ 
for any $1\leq i\leq m$ by induction such that 
\begin{itemize}
\item
$\sL_i|_{X_i}\simeq\sH_i$ and
\item
$\sL_i|_{\sX_{i-1}}\simeq\sL_{i-1}$ (if $i\geq 2$).
\end{itemize}
If $i=1$, then $\sL_1$ must be (isomorphic to) $\sH_1$. 
Assume that we have constructed $\sL_1,\dots,\sL_i$. 
Since
\[
\Pic(\sX_{i+1})\rightarrow\Pic(X_{i+1})\oplus\Pic(\sX_i)
\rightarrow\Pic(\sX_i\cap X_{i+1})
\]
is exact, it is enough to show $\sL_i|_{\sX_i\cap X_{i+1}}\simeq 
\sH_{i+1}|_{\sX_i\cap X_{i+1}}$ to construct $\sL_{i+1}$. 
We already know that the natural sequence 
\[
0\rightarrow\Pic(\sX_i\cap X_{i+1})\xrightarrow{\kappa}
\bigoplus_{j=1}^i\Pic(X_{j,i+1})
\]
is exact since $\sX_i\cap X_{i+1}$ has a unique minimal stratum. 
Both $\sL_i|_{\sX_i\cap X_{i+1}}$ and $\sH_{i+1}|_{\sX_i\cap X_{i+1}}$ 
map $\kappa$ to $(\sH_j|_{X_{j,i+1}})_j$, thus we can construct $\sL_{i+1}$. 

Therefore we obtain an invertible sheaf $\sL_m\in\Pic(\sX)$ 
such that $\sL|_{X_i}\simeq\sH_i$ 
for any $1\leq i\leq m$. Thus $\eta(\sL_m)=(\sH_i)_i$. 
\end{proof}

%\smallskip

\subsection{Snc Fano varieties and log Fano manifolds}\label{sncfano_section}

%\smallskip

We shall define simple normal crossing Fano varieties 
and log Fano manifolds. 

\begin{definition}[snc Fano varieties and log Fano manifolds]\label{sncfanopredfn}
\begin{enumerate}
\renewcommand{\theenumi}{\arabic{enumi}}
\renewcommand{\labelenumi}{(\theenumi)}
\item
A projective snc variety $\sX$ is said to be a 
\emph{simple normal crossing Fano variety} 
(\emph{snc Fano variety}, for short) 
if the dual of the dualizing sheaf $\omega_{\sX}^\vee$ is ample. 
\item
A projective log manifold $(X,D)$ is  said to be a 
\emph{log Fano manifold} if $-(K_X+D)$ is ample. 
\end{enumerate}
\end{definition}

\begin{ex}[See also Proposition \ref{indn}]\label{curve}
Let $(X, D)$ be a one-dimensional log Fano manifold. 
Then $X\simeq\pr^1$ and $D$ is either one point or empty, 
since $0<\deg (-(K_X+D))=2-2g-\deg D$ holds for the genus $g$ of $X$. 
Therefore, if $\sX$ is a one-dimensional 
snc Fano variety, then $\sX$ is isomorphic to either smooth or reducible conic. 
\end{ex}

We also define the \emph{index} and \emph{pseudoindex} 
for a simple normal crossing Fano variety 
and also for a log Fano manifold; whose notion is essential in the paper.  

\begin{definition}[index]\label{sncfanodfn}
\begin{enumerate}
\renewcommand{\theenumi}{\arabic{enumi}}
\renewcommand{\labelenumi}{(\theenumi)}
\item
Let $\sX$ be an snc Fano variety. We define the \emph{snc Fano index} of $\sX$ as 
\[
\max\{r\in\Z_{>0}\mid\omega_{\sX}^{\vee}\simeq\sL^{\otimes r}
\text{ for some }\sL\in\Pic(\sX)\}.
\]
\item
Let $(X, D)$ be a log Fano manifold. We define the \emph{log Fano index} 
of $(X, D)$ as 
\[
\max\{r\in\Z_{>0}\mid -(K_X+D)\sim rL
\text{ for some Cartier divisor }L\text{ on }X\}.
\]
\end{enumerate}
\end{definition}

\begin{definition}[pseudoindex]\label{psncfanodfn}
\begin{enumerate}
\renewcommand{\theenumi}{\arabic{enumi}}
\renewcommand{\labelenumi}{(\theenumi)}
\item
Let $\sX$ be an snc Fano variety. We define the \emph{snc Fano pseudoindex} 
of $\sX$ as 
\[
\min\{\deg_C(\omega_{\sX}^{\vee}|_C)\mid C\subset\sX
\text{ rational curve}\}.
\]
\item
Let $(X, D)$ be a log Fano manifold. We define the \emph{log Fano pseudoindex} 
of $(X, D)$ as 
\[
\min\{(-(K_X+D)\cdot C)\mid C\subset X
\text{ rational curve}\}.
\]
\end{enumerate}
\end{definition}

\begin{remark}\label{pseudo_geq}
For an snc Fano variety $\sX$ (resp.\ a log Fano manifold $(X, D)$), 
the snc Fano pseudoindex (resp.\ the log Fano pseudoindex) $\iota$ is divisible by the 
snc Fano index (resp.\ the log Fano index) $r$ by definition. 
In particular, $\iota\geq r$ holds. 
\end{remark}

\begin{remark}\label{fanoirrrmk}
Let $\sX$ be an $n$-dimensional snc Fano variety with the snc Fano index $r$, 
the snc Fano pseudoindex $\iota$
and $\sL$ be an invertible sheaf on $\sX$ such that 
$\omega_{\sX}^{\vee}\simeq\sL^{\otimes r}$ holds. 
Then $(X,D)$ is an $n$-dimensional log Fano manifold such that 
$-(K_X+D)\sim rL$ holds, 
where $(X,D)\subset\sX$ is an irreducible component with the conductor 
and $L$ is a divisor corresponding to the restriction of $\sL$ to $X$. 
It is easily shown by Remark \ref{adjunctionrmk}. Hence 
the log Fano index of $(X, D)$ is divisible by $r$ and the log Fano pseudoindex 
of $(X, D)$ is at least $\iota$.
\end{remark}

%\smallskip

\subsection{First properties of log Fano manifolds}\label{firstprop_section}

%\smallskip

We quickly give some properties about log Fano manifolds. 

\begin{thm}[{\cite[Theorem 1.3, 1.4]{Maeda}, \cite[Theorem 3.35]{KoMo}}]\label{cone}
Let $(X,D)$ be a log Fano manifold. 
Then $\NE(X)$ is spanned by a finite number of extremal rays. 
Furthermore, for any extremal ray $R\subset\NE(X)$, we have: 
\begin{itemize}
\item
The ray $R$ is spanned by a class of rational curve $C$ on $X$. 
\item
There exists a contraction morphism $\cont_R:X\rightarrow Y$ associated to $R$ and 
there exists an exact sequence 
\[
0\rightarrow\Pic(Y)\xrightarrow{\cont_R^*}\Pic(X)\xrightarrow{(\bullet\cdot C)}\Z.
\]
\end{itemize}
\end{thm}

\begin{lemma}[{\cite[Corollary 2.2, Lemma 2.3]{Maeda}}]\label{picard}
Let $(X,D)$ be a log Fano manifold. 
Then $\Pic(X)$ is torsion free. Furthermore if $\Bbbk=\C$, 
the homomorphism 
\[
\Pic(X)\rightarrow H^2(X^\text{\rm{an}}; \Z)
\]
is isomorphism. 
\end{lemma}

\begin{remark}\label{logfanopic_rmk}
Let $(X, D)$ be a log Fano manifold, $r$ be a positive integer and 
$L$ be a divisor on $X$ such that $-(K_X+D)\sim rL$. 
Then $L$ is uniquely defined up to linear equivalence by $X$, $D$ and $r$ 
by Lemma \ref{picard}. 
\end{remark}

\begin{remark}\label{sncpicrmk}
Let $\sX$ be an snc Fano variety. 
Then $\Pic(\sX)$ is torsion free by Lemma \ref{picard}, Theorem \ref{fujinothm} 
\eqref{fujinothm2} and Proposition \ref{picglue}. 
Hence we can also say the following: 
For an snc Fano variety $\sX$, $r$ a positive integer and $\sL$ an invertible sheaf 
such that $\omega_{\sX}^{\vee}\simeq\sL^{\otimes r}$, 
$\sL$ is uniquely defined up to isomorphism by $\sX$ and $r$. 
\end{remark}

\begin{proposition}\label{rhoonefano}
Let $(X, D)$ be a log Fano manifold with $\rho(X)=1$ and $D\neq 0$. 
Then $X$ is a Fano manifold whose Fano index $($resp.\ Fano pseudoindex$)$ 
is larger than the log Fano index $($resp.\ the log Fano pseudoindex$)$ of $(X, D)$. 
\end{proposition}

\begin{proof}
Since $-K_X\sim -(K_X+D)+D$ and $D$ is ample (we note that $\rho(X)=1$), 
the assertions are obvious. 
\end{proof}

\begin{thm}[cf. {\cite[Lemma 2.4]{Maeda}}]\label{fujinothm}
\begin{enumerate}
\renewcommand{\theenumi}{\arabic{enumi}}
\renewcommand{\labelenumi}{$(\theenumi)$}
\item\label{fujinothm1}
Let $(X,D)$ be a log Fano manifold such that 
the log Fano index is divisible by $r$ $($resp.\ the log Fano pseudo index $\geq\iota$$)$. 
Then $D$ is a $($connected$)$ snc Fano variety and the snc Fano index 
is also divisible by $r$ $($resp.\ the snc Fano pseudoindex $\geq\iota$$)$. 
\item\label{fujinothm2}
Let $\sX$ be an snc Fano variety. 
Then there is a unique minimal stratum of $\sX$. 
In particular, any two irreducible components of $\sX$ intersect with each other. 
\end{enumerate}
\end{thm}

\begin{proof}
\eqref{fujinothm1}
We know that $D$ is connected by \cite[Lemma 2.4 (a)]{Maeda}. 
Let $L$ be a divisor on $X$ such that $-(K_X+D)\sim rL$. 
Then $\omega_D^\vee\simeq\sO_X(-(K_X+D))|_D\simeq(\sO_X(L)|_D)^{\otimes r}$ 
by adjunction. 

\eqref{fujinothm2}
We can prove by using the same idea in \cite[Lemma 2.4 $(\rm{a}')$]{Maeda}. 
We remark that this is directly shown by \cite[Theorem 6.6 (ii)]{Ambro} 
and \cite[Theorem 3.47 (ii)]{fujino}. 
\end{proof}

\begin{corollary}\label{cpnt}
For an $n$-dimensional snc Fano variety $\sX$ with the snc Fano pseudoindex $\iota$, 
the number of irreducible components of $\sX$ is at most $n+2-\iota$. 
\end{corollary}

\begin{proof}
We can assume $\sX$ is reducible since the assertion is well-known 
for the irreducible case by \cite{mor79}. 
We prove by induction on $n$. 
If $n=1$, then $\sX$ is isomorphic to a reducible conic (see Example \ref{curve}). 
Hence the assertion is obvious. 

Now, we assume that the assertion holds for the case $n-1$. 
We know that any two irreducible components of $\sX$ 
intersect with each other by Theorem \ref{fujinothm} \eqref{fujinothm2}. 
Therefore, for any irreducible component with the conductor $(X,D)\subset\sX$, 
the number of irreducible components of $\sX$ minus one is 
equal to the number of irreducible components of $D$. 
We note that $D$ is an $(n-1)$-dimensional snc Fano variety 
such that the snc Fano index is at least $\iota$ 
by Remark \ref{fanoirrrmk} and Theorem \ref{fujinothm} \eqref{fujinothm1}. 
The number of irreducible components of $D$ is at most $n+1-\iota$ by induction step.  
Therefore the number of irreducible components of $\sX$ is at most $n+2-\iota$. 
\end{proof}

By Theorem \ref{fujinothm} \eqref{fujinothm2}, we also obtain the following 
corollary using Theorem \ref{glue} and Proposition \ref{picglue}.

\begin{corollary}\label{logfanopic}
Fix $n$, $r$, $m\in\Z_{>0}$. 
Let $(X_i, D_i)$ be an $n$-dimensional log Fano manifold whose log Fano index 
is divisible by $r$ for any $1\leq i\leq m$. 
Assume that the irreducible decomposition is written as 
$D_i=\sum_{j\neq i, 1\leq j\leq m}D_{ij}$ for any $1\leq i\leq m$, 
and there exist isomorphisms
\[
\phi_{ij}\colon D_{ij}\rightarrow D_{ji}
\]
for all distinct $1\leq i$, $j\leq m$ which satisfy the cocycle conditions 
\eqref{glue1} and \eqref{glue2} in Theorem \ref{glue}. 
Then there exists an $n$-dimensional snc Fano variety $\sX$ 
with the snc Fano index is divisible by $r$ whose 
irreducible decomposition can be written as 
$\sX=\bigcup_{i=1}^m(X_i, D_i)$.
\end{corollary}

Nowadays, thanks to the recent progress of minimal model program, 
we also know the following results. 

\begin{thm}[{\cite[Theorem 1]{zhang}}]\label{rat_conn}
Let $(X, D)$ be a log Fano manifold. 
Then $X$ is a rationally connected variety, that is, arbitrary two closed points in $X$ are 
joined by an irreducible rational curve. 
\end{thm}

\begin{thm}[{\cite[Corollary 1.3.2]{BCHM}}]\label{dream}
Let $(X, D)$ be a log Fano manifold. 
Then $X$ is a Mori dream space $($see \cite{HK} for definition$)$. 
\end{thm}

%\smallskip

\subsection{Bundles and subbundles}\label{bdle_section}

%\smallskip

In this subsection, we recall some bundle structures. 
The following lemma is well-known. 

\begin{lemma}\label{lemP}
Let $X$ be an irreducible variety, $D\subset X$ be an effective Cartier divisor 
and $c$ be a nonnegative integer. 
Let $\pi\colon X\rightarrow Y$ be a $\pr^c$-bundle such that 
$\pi|_D\colon D\rightarrow Y$ is a $\pr^{c-1}$-subbundle. 
That is, $\pi$ is a proper and smooth morphism such that 
$\pi^{-1}(y)\simeq\pr^c$ and $(\pi|_D)^{-1}(y)$ is isomorphic to a hyperplane section 
under this isomorphism for any closed point $y\in Y$. 
Then there exists a commutative diagram of $Y$-morphisms
\[
\begin{CD}
D                                  @>{\iota}>> X                                \\
@V{\iota_D}VV                                 @V{\iota_X}VV               \\
\pr_Y((\pi|_D)_*\sN_{D/X}) @>>j>         \pr_Y(\pi_*\sO_X(D)),
\end{CD}
\]
where 
\begin{itemize}
\item
$\iota$ is the inclusion map, 
\item
both $\iota_D$ and $\iota_X$ are isomorphisms, 
\item
$j$ is obtained by the natural surjection 
\[
\pi_*\sO_X(D)\rightarrow (\pi|_D)_*\sN_{D/X},
\]
where $\sN_{D/X}$ is the normal sheaf $\sO_D(D)$. 
\end{itemize}
Furthermore, we have $D\in|\sO_{\pr}(1)|$ under these isomorphisms, 
where $\sO_\pr(1)$ is the tautological invertible sheaf on $\pr_Y(\pi_*\sO_X(D))$. 
\end{lemma}

Next, we consider $\sQ^{c+1}$-bundles and $\sQ^c$-subbundles.  

\begin{definition}\label{Qq}
Let $\pi\colon X\rightarrow Y$ be a morphism between irreducible varieties 
and $c$ be a positive integer. 
We say that $\pi\colon X\rightarrow Y$ is a \emph{$\sQ^{c+1}$-bundle} if 
$\pi$ is a proper and flat morphism such that 
$\pi^{-1}(y)$ is (scheme theoretically) isomorphic to a hyperquadric in $\pr^{c+2}$. 

For a $\sQ^{c+1}$-bundle $\pi\colon X\rightarrow Y$ and 
an effective Cartier divisor $D$ on $X$, we say that 
$\pi|_D\colon D\rightarrow Y$  is a \emph{$\sQ^c$-subbundle} of $\pi$ if 
$(\pi|_D)^{-1}(y)$ is isomorphic to a hyperplane section under the isomorphism 
$\pi^{-1}(y)\simeq\sQ^{c+1}$ for any closed point $y\in Y$. 
We note that the morphisms $\pi$ and $\pi|_D$ is not needed to be smooth. 
(That is why we do not use the symbol $\Q$ but $\sQ$.)
\end{definition}

\begin{lemma}\label{lemQ}
Let $X$ be an irreducible variety, $D\subset X$ be an effective Cartier divisor, 
$Y$ be a smooth variety and $c$ be a positive integer. 
Suppose that $\pi\colon X\rightarrow Y$ is a $\sQ^{c+1}$-bundle and 
$\pi|_D\colon D\rightarrow Y$ is a $\sQ^c$-subbundle. 
Then we have: 
\begin{enumerate}
\renewcommand{\theenumi}{\roman{enumi}}
\renewcommand{\labelenumi}{\rm{(\theenumi)}}
\item\label{lemQ1}
The natural sequence 
\[
0\rightarrow\sO_Y\rightarrow\pi_*\sO_X(D)\rightarrow(\pi|_D)_*\sN_{D/X}\rightarrow 0
\]
is exact. 
\item\label{lemQ2}
$\pi_*\sO_X(D)$ and $(\pi|_D)_*\sN_{D/X}$ are locally free of rank $c+3$ and $c+2$, 
respectively. In particular, $P:=\pr_Y(\pi_*\sO_X(D))$ is a $\pr^{c+2}$-bundle over $Y$ 
and $H:=\pr_Y((\pi|_D)_*\sN_{D/X})$ is a $\pr^{c+1}$-subbundle. 
\item\label{lemQ3}
The natural homomorphism 
\[
\pi^*\pi_*\sO_X(D)\rightarrow\sO_X(D)
\]
is surjective, and it induces a relative quadric embedding 
$X\hookrightarrow P$ over $Y$. 
\item\label{lemQ4}
$D$ is isomorphic to the complete intersection $X\cap H$ in $P$ under these embeddings. 
\end{enumerate}
\end{lemma}

\begin{proof}
For any hyperquadric $\sQ^{c+1}$ in $\pr^{c+2}$, 
\[
h^0(\sQ^{c+1},\sO_{\sQ^{c+1}}(1))=c+3,\quad h^1(\sQ^{c+1},\sO_{\sQ^{c+1}}(1))=0
\]
holds. Hence we get the results of \eqref{lemQ1} and \eqref{lemQ2} 
by cohomology and base change theorem. 

The surjectivity of the homomorphism in \eqref{lemQ3} is easily shown 
since $\sO_{\sQ^{c+1}}(1)$ is generated by global sections. 
We can also show that this surjection induces 
a relative quadric embedding $X\hookrightarrow P$. Thus we have shown \eqref{lemQ3}. 

Similarly, we get a surjection 
\[
(\pi|_D)^*(\pi|_D)_*\sN_{D/X}\rightarrow\sN_{D/X}
\]
and this gives a relative quadric embedding $D\hookrightarrow H$. 
Then the composition $D\hookrightarrow H\subset P$ is equal to 
$D\subset X\hookrightarrow P$ by construction. 

Now we prove that $D=X\cap H$ in $P$ under these embeddings. 
Let $p\colon P\rightarrow Y$ be the projection. 
We note that for any closed point $y\in Y$ there exists a unique hyperplane $H^0_y\subset p^{-1}(y)$ containing $(\pi|_D)^{-1}(y)$. 
Indeed, if there exist two such distinct hyperplanes 
$H^1_y$, $H^2_y\subset p^{-1}(y)$ containing $(\pi|_D)^{-1}(y)$, 
then $(\pi|_D)^{-1}(y)$ is contained in the reduced $c$-dimensional linear subspace 
$H^1_y\cap H^2_y$, which is a contradiction. 
Let $H^0:=\bigcup_{y\in Y}H^0_y\subset P$. Then $H^0$ is the unique divisor of $P$ 
which is a $\pr^{c+1}$-subbundle containing $D$. 
Therefore, $H=H^0$ and $D=X\cap H$, since $D=X\cap H^0$ by construction. 
\end{proof}

\begin{lemma}\label{PQlem}
Let $X$ be an irreducible variety such that $h^1(X, \sO_X)=0$. 
\begin{enumerate}
\renewcommand{\theenumi}{\arabic{enumi}}
\renewcommand{\labelenumi}{$(\theenumi)$}
\item\label{PQlem1}
Let $c$ be a nonnegative integer and $p_1\colon X\times\pr^c\rightarrow X$, 
$p_2\colon X\times\pr^c\rightarrow\pr^c$ be the projections. 
Then 
\[
(p_1)_*(p_2^*\sO_{\pr^c}(1))\simeq\sO_X^{\oplus c+1}.
\]
\item\label{PQlem2}
Let $c\geq 2$ and  $p_1\colon X\times\Q^c\rightarrow X$, 
$p_2\colon X\times\Q^c\rightarrow\Q^c$ be the projections. 
Then 
\[
(p_1)_*(p_2^*\sO_{\Q^c}(1))\simeq\sO_X^{\oplus c+2}.
\]
\end{enumerate}
\end{lemma}

\begin{proof}
We prove both assertions by induction on $c$. 

\eqref{PQlem1}
The case $c=0$ is trivial. We assume that the assertion holds for the case: $c-1$. 
There has the canonical exact sequence 
\[
0\rightarrow\sO_{X\times\pr^c}\rightarrow p_2^*\sO_{\pr^c}(1)\rightarrow
(p_2|_{X\times\pr^{c-1}})^*\sO_{\pr^{c-1}}(1)\rightarrow 0.
\]
After taking $(p_1)_*$, the sequence 
\[ 
0\rightarrow\sO_X\rightarrow(p_1)_*(p_2^*\sO_{\pr^c}(1))\rightarrow
(p_1|_{X\times\pr^{c-1}})_*\bigl((p_2|_{X\times\pr^{c-1}})^*\sO_{\pr^{c-1}}(1)\bigr)\rightarrow 0
\]
is exact. We note that 
$(p_1|_{X\times\pr^{c-1}})_*((p_2|_{X\times\pr^{c-1}})^*\sO_{\pr^{c-1}}(1))\simeq\sO_X^{\oplus c}$ by the induction step. 
The sequence always splits since $h^1(X, \sO_X)=0$. 
Hence we have proved \eqref{PQlem1}. 

\eqref{PQlem2}
The case $c=2$ is the direct consequence of (1) 
since $\Q^2$ is isomorphic to $\pr^1\times\pr^1$. 
We assume that the assertion holds for the case: $c-1$. 
There has the canonical exact sequence 
\[
0\rightarrow\sO_{X\times\Q^c}\rightarrow p_2^*\sO_{\Q^c}(1)\rightarrow
(p_2|_{X\times\Q^{c-1}})^*\sO_{\Q^{c-1}}(1)\rightarrow 0.
\]
After taking $(p_1)_*$, we have the splitting exact sequence 
\[
0\rightarrow\sO_X\rightarrow(p_1)_*(p_2^*\sO_{\Q^c}(1))\rightarrow
\sO_X^{\oplus c+1}\rightarrow 0
\]
by repeating the same argument in the proof of \eqref{PQlem1}. 
Hence we have proved \eqref{PQlem2}. 
\end{proof}

%\smallskip

\subsection{Facts on extremal contractions and its applications}\label{extcont_section}

%\smallskip

In this section, we show the structure of the contraction morphism associated 
to a special ray using Wi\'sniewski's inequality. First, 
we give a criterion for a smooth projective variety 
to have the Picard number one in Lemma \ref{rhoone}. 
Second, we give some delicate structure properties for special log Fano manifolds 
in Proposition \ref{structure_cor}, which is essential to prove Theorems \ref{mukai0} 
and \ref{mukai1}. 

We remember Wi\'sniewski's inequality, which plays an essential role in this section. 

\begin{thm}[Wi\'sniewski's inequality \cite{wisn91}]\label{wisn_ineq}
Let $X$ be an $n$-dimensional smooth projective variety and 
$R\subset\overline{\NE}(X)$ be a $K_X$-negative extremal ray with 
the associated contraction morphism $\pi\colon X\rightarrow Y$. 
Then we have the inequality 
\[
\dim\Exc(\pi)+\dim F\geq n+l(R)-1
\]
for any nontrivial fiber $F$ of $\pi$. 
\end{thm}

We give a criterion for a smooth projective variety $X$ being $\rho(X)=1$ 
using Theorem \ref{wisn_ineq}. 

\begin{lemma}\label{rhoone}
Let $X$ be an irreducible smooth projective variety, 
$D\subset X$ be a prime divisor and 
$R\subset\overline{\NE}(X)$ be a $K_X$-negative extremal ray 
with the associated contraction morphism $\pi\colon X\rightarrow Y$ 
such that $(D\cdot R)>0$. 
\begin{enumerate}
\renewcommand{\theenumi}{\arabic{enumi}}
\renewcommand{\labelenumi}{$(\theenumi)$}
\item\label{rhoone1}
If the restriction morphism $\pi|_D:D\rightarrow\pi(D)$ is not birational, 
then $\pi$ is of fiber type, i.e., $\dim Y<\dim X$ holds. 
\item\label{rhoone2}
If $l(R)\geq 3$, then $\pi|_D:D\rightarrow Y$ is not a finite morphism.  
Furthermore, if $\rho(D)=1$ holds in addition, 
then $X$ is a Fano manifold with $\rho(X)=1$. 
\end{enumerate}
\end{lemma}

\begin{proof}
\eqref{rhoone1}
If $\pi$ is of birational type, then it is a divisorial contraction 
and the exceptional divisor is exactly $D$, 
since $\pi|_D:D\rightarrow\pi(D)$ is not birational. However, we get a contradiction 
since $(D\cdot R)>0$. Hence $\pi$ is of fiber type. 

\eqref{rhoone2}
Let us choose an arbitrary nontrivial fiber $F$ of $\pi$. 
We have $D\cap F\neq\emptyset$ since $(D\cdot R)>0$. 
Then 
\[
\dim(F\cap D)\geq\dim F-1\geq l(R)-2\geq 1
\]
by Wi\'sniewski's inequality (Theorem \ref{wisn_ineq}). 
Hence $F\cap D$ contains a curve. 
Now, we assume that the Picard number of $D$ is equal to one. 
Then $\pi(D)$ must be a point 
since all curves in $D$ are numerically proportional. 
Therefore $\pi$ is of fiber type by \eqref{rhoone1}. 
If $\dim Y\geq 1$, then $(D\cdot R)=0$; hence $Y$ is a point. 
In particular, $\rho(X)=1$. Thus $X$ is a Fano manifold since 
there exists a $K_X$-negative extremal ray. 
\end{proof}

We also show that there exists a `special' $K_X$-negative extremal ray 
for a log Fano manifold with nonzero boundary, 
which is essential to classify some special log Fano manifolds. 

\begin{lemma}\label{longray}
Let $(X,D)$ be a log Fano manifold with the log Fano index $r$ and 
the log Fano pseudoindex $\iota$, 
$L$ be a divisor on $X$ such that 
$-(K_X+D)\sim rL$ holds, and assume that $D\neq 0$.  
Then there exists an extremal ray $R\subset\NE(X)$ such that $(D\cdot R)>0$. 
Let $R$ be an extremal ray satisfying $(D\cdot R)>0$ and $\pi\colon X\rightarrow Y$ 
be the contraction morphism associated to $R$. 
Then $R$ is always $K_X$-negative and $l(R)\geq\iota+1$. 
Moreover, the restriction morphism 
\[
\pi|_{D_1}\colon D_1\rightarrow \pi(D_1)
\]
to its image is an algebraic fiber space, 
that is, $(\pi|_{D_1})_*\sO_{D_1}=\sO_{\pi(D_1)}$, 
for any irreducible component $D_1\subset D$. 
Furthermore, for a minimal rational curve $C\subset X$ of $R$, 
we have the following properties:
\begin{enumerate}
\renewcommand{\theenumi}{\arabic{enumi}}
\renewcommand{\labelenumi}{$(\theenumi)$}
\item\label{longray1}
If $l(R)=\iota+1$, then $(D\cdot C)=1$.
\item\label{longray2}
If $l(R)=r+2$ and $r\geq 2$, then $(L\cdot C)=1$ and 
$(D\cdot C)=2$.
\end{enumerate}
\end{lemma}

\begin{proof}
The existence of such extremal ray is obvious, since $D$ is a nonzero effective divisor 
and $\NE(X)$ is spanned by finite number of extremal rays. 
Let $R\subset\NE(X)$ be an extremal ray such that $(D\cdot R)>0$. 
Then $R$ is $K_X$-negative since $(-K_X\cdot R)=(-(K_X+D)\cdot R)+(D\cdot R)>0$. 

To see that $\pi|_{D_1}\colon D_1\rightarrow \pi(D_1)$ is an algebraic fiber space, 
it is enough to show that the homomorphism 
$\pi_*\sO_X\rightarrow(\pi|_{D_1})_*\sO_{D_1}$ is surjective. 
We know that the sequence 
\[
\pi_*\sO_X\rightarrow(\pi|_{D_1})_*\sO_{D_1}\rightarrow R^1\pi_*\sO_X(-D_1)
\]
is exact and $R^1\pi_*\sO_X(-D_1)=0$ by a vanishing theorem 
(see for example \cite[Theorem 2.42]{fujino}). Hence 
$\pi|_{D_1}\colon D_1\rightarrow \pi(D_1)$ is an algebraic fiber space for any 
irreducible component $D_1\subset D$. 

Let $C\subset X$ be a minimal rational curve of $R$. Then we have 
\[
l(R)=(-K_X\cdot C)=(-(K_X+D)\cdot C)+(D\cdot C)\geq\iota+1.
\]
If $l(R)=\iota+1$, then the above inequality is exactly equal. Hence $(D\cdot C)=1$ 
holds. 

If $l(R)=r+2$ and $r\geq 2$, then 
\[
r+2=l(R)=r(L\cdot C)+(D\cdot C)\geq r+1. 
\]
Therefore $(L\cdot C)=1$ and $(D\cdot C)=2$ holds.
\end{proof}

Using Lemma \ref{longray}, we can show a delicate structure properties 
for certain log Fano manifolds. 

\begin{proposition}\label{structure_cor}
Let $(X,D)$ be a log Fano manifold of the log Fano index $r$, log Fano pseudoindex 
$\iota$ and assume that $D\neq 0$. 
Pick an arbitrary extremal ray $R\subset\NE(X)$ 
such that $(D\cdot R)>0$ and let $\pi\colon X\rightarrow Y$ be 
the contraction morphism associated to $R$. 
Let $F$ be an arbitrary nontrivial fiber of $\pi$. 
Then $\dim(D\cap F)\geq\iota-1$ holds. 
Furthermore, we have the following results. 
\begin{enumerate}
\renewcommand{\theenumi}{\roman{enumi}}
\renewcommand{\labelenumi}{\rm{(\theenumi)}}
\item\label{structure_cor1}
If $\dim(D\cap F)=\iota-1$ for any nontrivial fiber $F$, then 
$\pi\colon X\rightarrow Y$ is a 
$\pr^\iota$-bundle and $\pi|_D\colon D\rightarrow Y$ is a $\pr^{\iota-1}$-subbundle.
\item\label{structure_cor2}
If $r\geq 2$ and there exists an irreducible component $D_1$ of $D$ such that 
$\dim(D_1\cap F)=r$ for any $F$, then 
one of the following holds. 
\begin{enumerate}
\renewcommand{\theenumii}{\alph{enumii}}
\renewcommand{\labelenumii}{\rm{(\theenumii)}}
\item\label{structure_cor2a}
$Y$ is a smooth projective variety and $\pi$ is the blowing up 
along an irreducible smooth projective subvariety $W\subset Y$ 
of codimension $r+2$.
\item\label{structure_cor2b}
$Y$ is smooth, $\pi\colon X\rightarrow Y$ is a $\sQ^{r+1}$-bundle and 
$\pi|_{D_1}\colon D_1\rightarrow Y$ is a $\sQ^r$-subbundle $($cf. Definition \ref{Qq}$)$. 
\item\label{structure_cor2c}
$\pi\colon X\rightarrow Y$ is a 
$\pr^{r+1}$-bundle and $\pi|_{D_1}\colon D_1\rightarrow Y$ is a $\pr^r$-subbundle.
\item\label{structure_cor2d}
$\pi_*\sO_X(L)$ is locally free of rank $r+2$, where $L$ is a divisor on $X$ 
such that $-(K_X+D)\sim rL$. Furthermore, 
$\pi\colon X\rightarrow Y$ is isomorphic to the projection 
$p\colon\pr_Y(\pi_*L)\rightarrow Y$
and $(\pi|_{D_1})^{-1}(y)$ is a hyperquadric section under the 
isomorphism $\pi^{-1}(y)\simeq\pr^{r+1}$ for any closed point $y\in Y$. 
Moreover, $\pi_*\sO_X(L)\simeq(p|_{D_1})_*(\sO_{\pr}(1)|_{D_1})$ under the isomorphism. 
\end{enumerate}
\end{enumerate}
\end{proposition}

\begin{proof}
Let $L$ be an ample divisor on $X$ such that $-(K_X+D)\sim rL$ and $C$ be a 
minimal rational curve of $R$. 
We note that $D$ and $F$ intersect with each other 
since $(D\cdot R)>0$. Hence 
\begin{equation}\label{long_ineq}
\dim(D\cap F)\geq\dim F-1\geq \dim X-\dim\Exc(\pi)+l(R)-2\\
\geq l(R)-2\geq\iota-1\geq r-1
\end{equation}
by Wi\'sniewski's inequality (Theorem \ref{wisn_ineq}) and by Lemma \ref{longray}. 

First, we consider the case \eqref{structure_cor1}. Then $\dim\Exc(\pi)=\dim X$ and $l(R)=\iota+1$. 
Hence $\pi$ is of fiber type, 
all fibers of $\pi$ are of dimension $\iota$, and the equalities 
$(D\cdot C)=1$ and $(-K_X\cdot C)=\iota+1$ hold by Lemma \ref{longray}. 
Therefore $\pi\colon X\rightarrow Y$ is a $\pr^\iota$-bundle and 
$\pi|_D\colon D\rightarrow Y$ is a $\pr^{\iota-1}$-subbundle by 
\cite[Lemma 2.12]{fujita}.

Next, we consider the case \eqref{structure_cor2}. 
We first show that $(D_1\cdot R)>0$. If not, any nontrivial fiber $F$ is included 
in $D_1$ (in particular, $\pi$ is of birational type). 
Then Wi\'sniewski's inequality (Theorem \ref{wisn_ineq}) 
and Lemma \ref{longray} shows that 
\[
\dim\Exc(\pi)+r\geq\dim X+l(R)-1\geq\dim X+\iota. 
\]
Hence $\pi$ is of fiber type, this leads to a contradiction. 
Consequently, we have $(D_1\cdot R)>0$. 

We first assume that $\dim\Exc(\pi)<\dim X$. Then 
$\dim\Exc(\pi)=\dim X-1$ and $l(R)=r+1$ by substituting $D_1$ for (\ref{long_ineq}).
Hence $\pi$ is a divisorial contraction 
such that $\dim F=r+1$ for any $F$, 
and the equality $(D\cdot C)=1$ holds by Lemma \ref{longray} \eqref{longray1}. 
Thus, $Y$ is a smooth projective variety and $\pi$ is the blowing up 
whose center $W\subset Y$ is a smooth projective subvariety of 
codimension $r+2$ 
by \cite[Theorem 4.1 (iii)]{AW}. Therefore the condition \eqref{structure_cor2a} is 
satisfied in the case $\dim\Exc(\pi)<\dim X$. 

We second consider the case where $\dim\Exc(\pi)=\dim X$, that is, 
$\pi$ is of fiber type. We note that $l(R)=r+1$ or $r+2$ by (\ref{long_ineq}). 

We consider the case where $\pi$ is of fiber type and $l(R)=r+1$. 
Then $\dim F=r+1$ for any fiber and 
the equalities $(D_1\cdot C)=1$ 
and $(-K_X\cdot C)=r+1$ hold by (\ref{long_ineq}) and Lemma \ref{longray} 
\eqref{longray1}. 
Thus $\pi_*\sO_X(D_1)$ is locally free of rank $r+3$ and $X$ is embedded over $Y$ 
into $\pr_Y(\pi_*\sO_X(D_1))$
as a divisor of relative degree $2$ by \cite[Theorem B]{ABW}. 
Therefore the condition \eqref{structure_cor2b} satisfied in the case 
$\dim\Exc(\pi)=\dim X$ and $l(R)=r+1$. 

We consider the case where $\pi$ is of fiber type and $l(R)=r+2$. 
Then $(L\cdot C)=1$ and either $(D_1\cdot C)=1$ or $2$ holds 
by Lemma \ref{longray}. Thus $\pi\colon X\rightarrow Y$ is isomorphic to the 
$\pr^{r+1}$-bundle $\pr_Y(\pi_*\sO_X(L))$ by \cite[Lemma 2.12]{fujita}. 
If $(D_1\cdot C)=1$, then $\pi|_{D_1}\colon D_1\rightarrow Y$ is a $\pr^r$-subbundle. 
Therefore the condition \eqref{structure_cor2c} satisfied in the case 
$\dim\Exc(\pi)=\dim X$, $l(R)=r+2$ and  $(D_1\cdot C)=1$.  

Finally, we consider the remained case where 
$\pi$ is of fiber type, $l(R)=r+2$ and $(D_1\cdot C)=2$. 
Under the isomorphism 
$X\simeq\pr_Y(\pi_*\sO_X(L))$, we have a natural exact sequence 
\[
0\rightarrow\sO_{\pr}(1)(-D_1)\rightarrow\sO_{\pr}(1)
\rightarrow\sO_{\pr}(1)|_{D_1}\rightarrow 0,
\]
where $\sO_\pr(1)$ is the tautological invertible sheaf on $\pr_Y(\pi_*\sO_X(L))$. 
After taking $p_*$, we also obtain an exact sequence 
\[
0\rightarrow 0\rightarrow\pi_*\sO_X(L)
\rightarrow(p|_{D_1})_*(\sO_{\pr}(1)|_{D_1})\rightarrow 0
\]
by cohomology and base change theorem, since $h^i(\pr^{r+1}, \sO(-1))=0$ holds for $i=0,1$. 
Therefore the condition \eqref{structure_cor2d} satisfied in the case 
$\dim\Exc(\pi)=\dim X$, $l(R)=r+2$ and  $(D_1\cdot C)=2$.  
\end{proof}

%\smallskip

\subsection{Properties on scrolls}\label{property_scroll}

%\smallskip

In Section \ref{property_scroll}, we consider special toric varieties which are 
the projective space bundles of which splits into invertible sheaves 
over projective spaces, so called the Hirzebruch-Kleinschmidt varieties. 
During Section \ref{property_scroll}, we fix the notation. 

\begin{notation}\label{scroll_not}
Let $s$, $t$ be positive integers and $a_0,\dots,a_t$ be integers with 
$0=a_0\leq a_1\leq\dots\leq a_t$. 
Let 
\[
X:=\pr[\pr^s; a_0,\dots,a_t],
\]
that is, 
\[
X=\pr_{\pr^s}(\sO(a_0)\oplus\dots\oplus\sO(a_t)).
\]
We also let 
\[
D_i:=\pr[\pr^s; a_0,\dots,a_{i-1},a_{i+1},\dots,a_t]\cansub X=\pr[\pr^s; a_0,\dots,a_t],
\]
that is, the embedding is obtained by the canonical projection, for any $0\leq i\leq t$. 
(See Notation and terminology in Section \ref{introsection}.) 
\end{notation}

\begin{lemma}\label{scroll_lem}
We have the following properties. 
\begin{enumerate}
\renewcommand{\theenumi}{\arabic{enumi}}
\renewcommand{\labelenumi}{$(\theenumi)$}
\item\label{scroll_lem1}
$\Pic(X)=\Z[\sO(1; 0)]\oplus\Z[\sO(0; 1)]$.
\item\label{scroll_lem2}
$\sO_X(-K_X)\simeq\sO(s+1-\sum_{i=1}^ta_i;\,\, t+1)$. 
\item\label{scroll_lem3}
$D_i\in|\sO(-a_i; 1)|$ for any $0\leq i\leq t$.
\item\label{scroll_lem35}
$\deg_{C_f}(\sO(u; v)|_{C_f})=v$ and $\deg_{C_h}(\sO(u; v)|_{C_h})=u$, 
where $C_f$ is a line in a fiber of $X\rightarrow\pr^s$ and 
$C_h$ is a line in $\pr[\pr^s; a_0]\cansub X=\pr[\pr^s; a_0,\dots,a_t].$
\item\label{scroll_lem4}
$\Nef(X)=\R_{\geq 0}[\sO(1; 0)]+\R_{\geq 0}[\sO(0; 1)]$ and 
$\Eff(X)=\R_{\geq 0}[\sO(1; 0)]+\R_{\geq 0}[\sO(-a_t; 1)].$
\item\label{scroll_lem5}
For a divisor $D=\sum_{i=1}^tc_iD_i+dH$ with $c_i, d\in\Z$, where $H$ is the 
pullback of a hyperplane in $\pr^s$, the value $h^0(X, \sO_X(D))$ is exactly 
the number of the elements of the set 
\begin{eqnarray*}
\Biggl\{(P_1,\dots,P_s,Q_1,\dots,Q_t)\in\Z^{\oplus s+t}\,\Bigg| \,
-\sum_{j=1}^tQ_j\geq 0,\,\, Q_i\geq-c_i\ (1\leq i\leq t),\\
-\sum_{i=1}^sP_i+\sum_{j=1}^ta_jQ_j\geq -d,\,\, P_1,\dots,P_s\geq 0\Biggr\}.
\end{eqnarray*}
\item\label{scroll_lem6}
If there exists an effective divisor $D$ on $X$ with 
$D\in|\sO(k; 1)|$ such that $k<-a_{t-1}$ holds, 
then $D$ always contains $D_t$ as an irreducible component. 
%\item\label{scroll_lem65}
%Assume that $a_t>0$. Then we have 
%$h^0(X, \sO(-a_t; 1))=1+h$, 
%where $h$ is the number of the elements of the set $\{1\leq j\leq t-1|a_j=a_t\}$. 
\item\label{scroll_lem7}
If a member $D\in|\sO(k; 2)|$ 
is reduced, then $k\geq -a_t-a_{t-1}$. 
\item\label{scroll_lem8}
Assume that $a_{t-2}<a_t$. Then any effective and reduced 
divisor $D$ on $X$ with 
$D\in|\sO(-a_t-a_{t-1}; 2)|$ is decomposed into two 
irreducible components 
$D^t$ and $D^{t-1}$ such that $D^t\sim D_t$ and $D^{t-1}\sim D_{t-1}$: 
here $D^t=D_t$ if $a_{t-1}<a_t$. 
Furthermore, after taking an automorphism of $X$ over $\pr^s$, 
$D^t$ and $D^{t-1}$ can move $D_t$ and $D_{t-1}$, respectively.
%Furthermore, we have 
%\begin{eqnarray*}
%h^0(X, \sO(-a_t-a_{t-1}; 2))=l+2+
%\binom{a_t-a_{t-1}+s}{s}\\
%+\left\{
%\begin{array}{ll}
%1 & (\text{if } a_{t-1}=0)\\
%0 & (\text{if } a_{t-1}>0)\\
%\end{array}
%\right\}
%+\left\{
%\begin{array}{ll}
%1 & (\text{if } a_{t-1}=a_t)\\
%0 & (\text{if } a_{t-1}<a_t)\\
%\end{array}
%\right\},
%\end{eqnarray*}
%\where $l$ is the number of the elements of the set $\{1\leq j\leq t-1|a_j=a_{t-1}\}$. 
\end{enumerate}
\end{lemma}

\begin{proof}
Since \eqref{scroll_lem1}--\eqref{scroll_lem5} are well-known, 
we shall prove \eqref{scroll_lem6}--\eqref{scroll_lem8}. 
We note that the total coordinate ring of $X$ is 
the $\Z^{\oplus 2}$-graded polynomial ring 
\[
\Bbbk[x_0, \dots, x_s, y_0,\dots,y_t]
\]
with the grading
\[
\deg x_i=(1, 0),\quad\deg y_j=(-a_j, 1).
\]
First, we consider \eqref{scroll_lem6}. A defining equation of $D$ 
is expressed as a linear combination of the monomials 
\[
y_t\prod_{i=0}^sx_i^{m_i}
\]
with $\sum_{i=0}^sm_i=k+a_t$, since $k<a_{t-1}$. 
Thus $D$ contains a component of the subvariety 
defined by $y_t=0$, which is nothing but $D_t$. 

Next, we consider \eqref{scroll_lem7}. 
If $k<-a_t-a_{t-1}$, then a defining equation of $D$ 
is expressed as a linear combination of the monomials 
\[
y_t^2\prod_{i=0}^sx_i^{m_i}
\]
with $\sum_{i=0}^sm_i=k+2a_t$. 
Therefore $D$ must be nonreduced since $D-2D_t$ is an effective divisor. 

Finally, we consider \eqref{scroll_lem8}. 
If $a_{t-1}<a_t$, then a defining equation of $D$ 
is expressed as a linear combination of  the monomials 
\[
y_t^2\prod_{i=0}^sx_i^{m_i}
\]
with $\sum_{i=0}^sm_i=a_t-a_{t-1}$ and
\[
y_{t-j}y_t
\]
with $1\leq j\leq l$, where $l$ is the number of the elements of the set 
$\{1\leq j\leq t-1\mid a_j=a_{t-1}\}$. 
Then $D$ has a component of the subvariety 
defined by $y_t=0$, which is nothing but $D_t$. We can also show that $D^{t-1}:=D-D_t$ is 
irreducible if $D$ is reduced and $D^{t-1}\sim D_{t-1}$. 
If $a_{t-2}<a_{t-1}=a_t$, a defining equation of $D$ is expressed as 
a linear combination of monomials 
\[
y_{t-1}^2,\quad y_{t-1}y_t\quad\text{and}\quad y_t^2.
\]
Hence $D$ has exactly two irreducible components $D^t$ and $D^{t-1}$ 
such that $D^t\sim D^{t-1}\sim D_t(\sim D_{t-1})$ since $D$ is reduced. 
The existence of an automorphism $\alpha$ of $X$ over $p\colon X\rightarrow\pr^s$ 
such that $\alpha(D^t)=D_t$ and $\alpha(D^{t-1})=D_{t-1}$ 
is easy in either case by a simple linear-algebraic argument. 
\end{proof}

\begin{corollary}\label{scroll_cor1}
Let $D$ be a member of 
$D\in|\sO(c; d)|$ for some $d>0$. 
Assume that 
$(X, D)$ is a log Fano manifold such that 
the log Fano index is $r$ and the log Fano pseudoindex is $\iota$. 
\begin{enumerate}
\renewcommand{\theenumi}{\arabic{enumi}}
\renewcommand{\labelenumi}{$(\theenumi)$}
\item\label{scroll_cor11}
If $\iota\geq t$, then $d=1$, $t=\iota$ and $s\geq\iota-1$ holds. 
Furthermore, if $s=\iota-1$, then $a_1=\cdots=a_{\iota-1}=0$ and $c=-a_\iota$.
\item\label{scroll_cor12}
If $r\geq t$ $($hence $\iota\geq t$ holds$)$, $s=r$ and $r\geq 2$, then we have 
$r=\iota$ and either $(a_1,\dots,a_{r-2},a_{r-1},c)=(0,\dots,0,0,1-a_r)$ 
or $(0,\dots,0,1,-a_r)$ holds. 
\end{enumerate}
\end{corollary}

\begin{proof}
By Lemma \ref{scroll_lem} \eqref{scroll_lem2}, 
\[
\sO_X(-(K_X+D))\simeq\sO\Biggl(s+1-\sum_{i=1}^ta_i-c;\,\,\,t+1-d\Biggr).
\]
Hence $t+1-d\geq\iota$ by Lemma \ref{scroll_lem} \eqref{scroll_lem35}. 
Thus $d=1$ and $t=\iota$ if $\iota\geq t$ holds 
(resp.\ $d=1$ and $t=\iota=r$ if $r\geq t$ holds). 
We also note that 
$s+1-\sum_{i=1}^\iota a_i-c$ is at least $\iota$ and 
is a positive multiple of $r$ and $c\geq -a_\iota$. 
Hence $s\geq\iota-1+\sum_{i=1}^{\iota-1}a_i\geq\iota-1$. 

\eqref{scroll_cor11}
If $s=\iota-1$, then $\iota-\sum_{i=1}^\iota a_i\geq\iota+c\geq\iota-a_\iota$. Therefore 
$\sum_{i=1}^{\iota-1}a_i=0$ and $c=-a_\iota$ hold. 

\eqref{scroll_cor12}
If $s=r$ and $r\geq 2$, then $r+1-\sum_{i=1}^ra_i-c$ is divisible by $r$ and 
$\sum_{i=1}^ra_i+c\geq\sum_{i=1}^{r-1}a_i\geq 0$. 
Hence $\sum_{i=1}^ra_i+c=1$. 
Therefore either $(a_1,\dots,a_{r-2},a_{r-1},c)=(0,\dots,0,0,1-a_r)$ 
or $(0,\dots,0,1,-a_r)$ holds. 
\end{proof}

\begin{corollary}\label{scroll_cor2}
Let $r:=t-1$ with $r\geq 2$ and $D$ be a member of 
$D\in|\sO(c; d)|$ for some $d>0$. 
Assume that 
$(X, D)$ is a log Fano manifold such that 
the log Fano index is divisible by $r$. 
Then we have $d=2$ and $s\geq r-1$. Furthermore, if $s=r-1$, 
then $a_1=\cdots=a_{r-1}=0$ and $c=-a_r-a_{r+1}$.
\end{corollary}

\begin{proof}
We repeat the argument similar to Corollary \ref{scroll_cor1}. 
By Lemma \ref{scroll_lem} \eqref{scroll_lem2}, 
\[
\sO_X(-(K_X+D))\simeq 
\sO\Biggl(s+1-\sum_{i=1}^{r+1}a_i-c;\,\,\, r+2-d\Biggr).
\]
Thus we have $d=2$ since $r+2-d$ is a positive multiple of $r$ 
and $r\geq 2$. 
We also know that $s\geq r-1+\sum_{i=1}^{r+1}a_i+c\geq r-1+\sum_{i=1}^{r-1}a_i\geq r-1$ 
by Lemma \ref{scroll_lem} \eqref{scroll_lem7}. 
Furthermore, if $s=r-1$, then 
$r-\sum_{i=1}^{r+1}a_i\geq r+c\geq r-a_r-a_{r+1}$. 
Therefore we complete the proof. 
\end{proof}

%\medskip

\section{Examples}\label{ex_section}

%\smallskip

In this section, we give some examples of log Fano manifolds with large log Fano indices.

\subsection{Example of dimension $2\iota-1$ and 
log Fano (pseudo)index $\iota$}\label{sbsc_mukai0}

First, we consider the case \eqref{scroll_cor11} in Corollary \ref{scroll_cor1}, 
which is the important example of $(2\iota-1)$-dimensional log Fano manifold 
with the log Fano (pseudo)index $\iota$ (See Theorem \ref{mukai0}). 

\begin{exmuk}\label{2rminusone}
Let $\iota\geq 2$, $m\geq 0$, 
\[
X=\pr[\pr^{\iota-1}; 0^\iota, m]\quad\text{and}\quad D\in|\sO(-m; 1)|.
\]
We know that $\sO(1; 1)$ is an ample 
invertible sheaf and $\sO_X(-(K_X+D))\simeq\sO(1; 1)^{\otimes\iota}$. 

If $m>0$, then $D$ is unique and $D=\pr[\pr^{\iota-1}; 0^\iota]\cansub 
X=\pr[\pr^{\iota-1}; 0^\iota, m]$ 
by Lemma \ref{scroll_lem} \eqref{scroll_lem6}. 

If $m=0$, then $X=\pr^{\iota-1}\times\pr^\iota$ and 
$D\in|\sO_{\pr^{\iota-1}\times\pr^\iota}(0,1)|$. 
Hence any member $D\in|\sO_{\pr^{\iota-1}\times\pr^\iota}(0,1)|$ 
is always an irreducible smooth divisor and 
can move $\pr[\pr^{\iota-1}; 0^\iota]\cansub\pr[\pr^{\iota-1}; 0^\iota, m]$ 
after taking an automorphism of $X$ over $\pr^{\iota-1}$. 
We note that the dimension of $|\sO_{\pr^{\iota-1}\times\pr^\iota}(0,1)|$ 
is equal to $\iota$ 
by Lemma \ref{scroll_lem} \eqref{scroll_lem5}. 

Therefore $(X, D)$ is a $(2\iota-1)$-dimensional log Fano manifold 
with the log Fano index $\iota$ and the log Fano pseudoindex $\iota$ 
for any $D\in|\sO(-m; 1)|$. 
\end{exmuk}

%\smallskip

\subsection{Examples of dimension $2r$ and log Fano index $r$}\label{sbsc_mukai1}

Next, we give examples of $2r$-dimensional log Fano manifolds 
with the log Fano indices $r$ (See Theorem \ref{mukai1}).

\begin{example}\label{burouappu}
Let $X:=\Bl_{\pr^{r-2}}\pr^{2r}\xrightarrow{\Bl}\pr^{2r}$, that is, the blowing up of $\pr^{2r}$ along 
a linear subspace $\pr^{r-2}$ of dimension $r-2$. 
Let $E\subset X$ be the exceptional divisor. 
Consider the linear system $|\Bl^*\sO_{\pr^{2r}}(1)\otimes\sO_X(-E)|$. 
It is easy to show that the linear system is of dimension $r+1$ and any element 
$D$ in the system is the strict transforms of a hyperplane in $\pr^{2r}$ containing 
the center of the blowing up. In particular, $D$ is irreducible and smooth. 
The invertible sheaf $\sH:=\Bl^*\sO_{\pr^{2r}}(2)\otimes\sO_X(-E)$ is ample. 
We also know that 
$\sO_X(-(K_X+D))\simeq\sH^{\otimes r}$. Therefore 
$(X, D)$ is a $2r$-dimensional log Fano manifold with the log Fano index $r$ 
for any $D\in|\Bl^*\sO_{\pr^{2r}}(1)\otimes\sO_X(-E)|$. 
\end{example}

\begin{example}\label{pP}
Let $X:=\pr^{r-1}\times\pr^{r+1}$ and $D$ is an effective divisor on $X$ such that 
$D\in|\sO_{\pr^{r-1}\times\pr^{r+1}}(0,2)|$. 
Then the dimension of the linear system is $(r+2)(r+3)/2-1$ 
and $D$ is a simple normal crossing divisor if and only if $D$ is the pull back of 
the smooth or reducible hyperquadric in $\pr^{r+1}$. In particular, a general element 
in the linear system is a simple normal crossing divisor. 
Let $\sH:=\sO_{\pr^{r-1}\times\pr^{r+1}}(1,1)$. Then $\sH$ 
is an ample invertible sheaf on $X$ 
and $\sO_X(-(K_X+D))\simeq\sH^{\otimes r}$. Therefore 
$(X, D)$ is a $2r$-dimensional log Fano manifold with the log Fano index $r$ 
for any simple normal crossing $D\in|\sO_{\pr^{r-1}\times\pr^{r+1}}(0,2)|$. 
\end{example}

\begin{example}\label{kayaku}
Let 
\[
X:=\pr[\pr^{r-1}; 0^r, m_1, m_2]
\]
with $0\leq m_1\leq m_2$ and $1\leq m_2$, 
and $D$ is an effective divisor on $X$ such that 
$D\in|\sO(-m_1-m_2; 2)|$. 
All reduced elements in this 
linear system can be seen the sum of 
\[
\pr[\pr^{r-1}; 0^r, m_1],\,\,\, \pr[\pr^{r-1}; 0^r, m_2]\,\,\, 
\cansub X=\pr[\pr^{r-1}; 0^r, m_1, m_2] 
\]
after taking an automorphism over $\pr^{r-1}$ 
by Lemma \ref{scroll_lem} \eqref{scroll_lem8}. 
In particular, this must be a simple normal crossing divisor. 
We note that the dimension of the linear system 
$|\sO(-m_1-m_2; 2)|$ is equal to 
\[
\left\{ \begin{array}{ll}
2 & (m_1=m_2) \\
\binom{m_2+r-1}{r-1}+r-1 & (m_1=0) \\
\binom{m_2-m_1+r-1}{r-1} & (0<m_1<m_2)
\end{array} \right.
\]
by Lemma \ref{scroll_lem} \eqref{scroll_lem5}. 
Let $\sH:=\sO(1; 1)$. 
Then $\sH$ is an ample invertible sheaf on $X$ 
and $\sO_X(-(K_X+D))\simeq\sH^{\otimes r}$. Therefore 
$(X, D)$ is a $2r$-dimensional log Fano manifold with the log Fano index $r$ 
for any reduced $D\in|\sO(-m_1-m_2; 2)|$. 
\end{example}

\begin{example}[See also Remark \ref{fanoQrmk}]\label{fanoQ}
Let 
\[
E:=\pr[\pr^{r-1}; 0^{r+1}]\cansub X':=\pr[\pr^{r-1}; 0^{r+1}, m]
\]
with $m\geq 0$. We note that $E\simeq\pr^{r-1}\times\pr^r$. 
Consider a smooth divisor $B$ in $X'$ with $B\in|\sO(0; 2)|$ 
such that the intersection $B\cap E$ is also smooth. 
We note that the homomorphism 
\[
H^0(X', \sO(0; 2))\rightarrow H^0(E, \sO(0; 2)|_E)
\]
is surjective since $H^1(X', \sO(0; 2)(-E))
%=H^1(\pr, \sO_\pr(K_\pr)\otimes p^*\sO_{\pr^{r-1}}(r)\otimes\sO_\pr(r+3))
=0$ 
by Kodairs's vanishing theorem. 
Hence general $B\in|\sO(0; 2)|$ satisfies this property. 
We also note that the dimension of the linear system $|\sO(0; 2)|$ is equal to 
\[
\binom{2m+r-1}{r-1}+(r+1)\binom{m+r-1}{r-1}+\frac{(r+1)(r+2)}{2}-1
\]
by Lemma \ref{scroll_lem} \eqref{scroll_lem5}. 
Let $\tau\colon X\rightarrow X'$ be the double cover of $X'$ 
with the branch divisor $B$, and $D$ be the strict transform of $E$ on $X$. Then 
$X$ is smooth and $D\simeq\pr^{r-1}\times\Q^r$ by construction. 
We know that $\sO_X(-K_X)\simeq\tau^*(\sO_{X'}(-K_{X'})\otimes\sO(0; -1))
\simeq\tau^*\sO(r-m; r+2)$ and 
$\sO_X(D)\simeq\tau^*\sO(-m; 1)$. 
Let $\sH:=\tau^*\sO(1; 1)$ an ample invertible sheaf 
on $X$. Then 
$\sO_X(-(K_X+D))\simeq\sH^{\otimes r}$. 
We also note that $\sH$ cannot be divisible by any positive number larger than one by 
Remark \ref{fanoQrmk}. 
Therefore 
$(X, D)$ is a $2r$-dimensional log Fano manifold with the log Fano index $r$. 
\end{example}

\begin{example}[See also Remark \ref{rthreermk}]\label{rthree}
In this example, we consider the case $r\geq 3$. Let 
\[
D:=\pr[\Q^r; 0^r]\cansub X:=\pr[\Q^r; 0^r, m]
\]
with $m\geq 0$. We note that $D$ is isomorphic to $\pr^{r-1}\times\Q^r$. 
We also note that there exists a unique element in $|D|$ if $m>0$. 
If $m=0$ then $X=\pr^r\times\Q^r$ and $D\in|\sO_{\pr^r\times\Q^r}(1, 0)|$ 
hence the dimension of the linear system $|D|$ is equal to $r$; any element in $|D|$ 
defines a smooth divisor in $X$. 
Let $\sH:=\sO(1; 1)$. 
We can show that $\sH$ is an ample invertible sheaf on $X$ provided that $m\geq 0$, 
and $\sO_X(-(K_X+D))\simeq\sH^{\otimes r}$ by an easy calculation. 
Therefore 
$(X, D)$ is a $2r$-dimensional log Fano manifold with the log Fano index $r$. 
\end{example}

\begin{example}\label{rtwo}
In this example, we only consider the case $r=2$. Let 
\[
D:=\pr_{\pr^1\times\pr^1}(\sO^{\oplus 2})\cansub
X:=\pr_{\pr^1\times\pr^1}(\sO^{\oplus 2}
\oplus\sO(m_1, m_2))
\]
with $0\leq m_1\leq m_2$. 
We note that there exists a unique element in $|D|$ if $m_2>0$. 
If $m_1=m_2=0$ then $X=\pr^1\times\pr^1\times\pr^2$ and 
$D\in|\sO_{\pr^1\times\pr^1\times\pr^2}(0,0,1)|$ 
hence the dimension of the linear system $|D|$ is equal to $2$; any element 
in $|D|$ defines a smooth divisor. 
Let $\sH:=p^*\sO_{\pr^1\times\pr^1}(1,1)\otimes\sO_\pr(1)$, 
where $p\colon X\rightarrow\pr^1\times\pr^1$ is the projection 
and $\sO_\pr(1)$ is the tautological invertible sheaf with respect to the projection $p$. 
We can show that $\sH$ is an ample invertible sheaf on $X$ 
provided that $0\leq m_1\leq m_2$, and 
$\sO_X(-(K_X+D))\simeq\sH^{\otimes 2}$ by an easy calculation. Therefore 
$(X, D)$ is a $4$-dimensional log Fano manifold with the log Fano index $2$. 
\end{example}

\begin{example}[See also Remark \ref{Tprmk}]\label{Tp}
Let 
\[
D:=\pr_{\pr^r}(T_{\pr^r})\cansub
X:=\pr_{\pr^r}(T_{\pr^r}\oplus\sO(m))
\]
with $m\geq 1$. 
We first note that $D$ is unique in its linear system $|D|$ if $m\geq 2$. 
If $m=1$ then the dimension of $|D|$ is equal to $r+1$. This is easy from the exact sequence 
\[
0\rightarrow\sO_X\rightarrow\sO_X(D)\rightarrow\sN_{D/X}\rightarrow 0
\]
and the fact $\sN_{D/X}\simeq\sO_{\pr^r\times\pr^r}(1-m,1)|_D$ under an embedding 
$D\subset\pr^r\times\pr^r$ of bidegree $(1, 1)$. 
We note that there exists an embedding 
$X\subset X_1:=\pr[\pr^r; 1^{r+1}, m]$ 
obtained by the surjection $\alpha$ in the exact sequence 
\[0\rightarrow\sO_{\pr^r}\rightarrow\sO(1)^{\oplus r+1}\xrightarrow{\alpha}
T_{\pr^r}\rightarrow 0.
\]
Let $\sH:=\sO(0; 1)$ on $X_1$. Then $\sH$ is an ample invertible sheaf on $X_1$ 
provided that $m\geq 1$, and satisfies 
$\sO_X(-(K_X+D))\simeq(\sH|_X)^{\otimes r}$. Therefore 
$(X, D)$ is a $2r$-dimensional log Fano manifold with the log Fano index $r$. 
\end{example}

\begin{example}\label{Pp}
Let $X:=\pr^r\times\pr^r$ and $D$ is an effective divisor on $X$ with 
$D\in|\sO_{\pr^r\times\pr^r}(1,1)|$. 
Then the dimension of the linear system is $r(r+2)$, any smooth element 
is isomorphic to $\pr_{\pr^r}(T_{\pr^r})$ and any non-smooth element 
is the union of the first and second pullbacks of hyperplanes. 
In particular, any $D$ in the linear system $|\sO_{\pr^r\times\pr^r}(1,1)|$ 
is a simple normal crossing divisor. 
Let $\sH:=\sO_{\pr^r\times\pr^r}(1,1)$. Then $\sH$ is an ample invertible sheaf and  $\sO_X(-(K_X+D))\simeq\sH^{\otimes r}$. Therefore 
$(X, D)$ is a $2r$-dimensional log Fano manifold with the log Fano index $r$ 
for any $D\in|\sO_{\pr^r\times\pr^r}(1,1)|$. 
\end{example}

\begin{example}\label{zerozeroone}
Let 
\[
X:=\pr[\pr^r; 0^r, 1]. 
\]
We can view $X$ as the blowing up of $P:=\pr^{2r}$ along 
a linear subspace $H\subset P$ of dimension $r-1$. 
Let $\phi\colon X\rightarrow P$ be the blowing up 
and $E$ be the exceptional divisor of $\phi$. Then 
\[
E=\pr[\pr^r; 0^r]\cansub X=\pr[\pr^r; 0^r, 1].
\]
Let $D$ be an effective divisor such that $D\in|\sO(0; 1)|$. 
Any smooth element in the linear system $|\sO(0; 1)|$ corresponds to 
the strict transform of a hyperplane in $P$ which does not contain $H$. 
Any non-smooth element in the linear system $|\sO(0; 1)|$ 
can be written as $E+D_0$, where $D_0$ is the strict transform of a hyperplane 
in $P$ which contains $H$. In particular, any divisor in the linear system $|\sO(0; 1)|$ 
is a simple normal crossing divisor. 
We also note that the dimension of the linear system $|\sO(0; 1)|$ is equal to $2r$. 
Let $\sH:=\sO(1; 1)$. 
Then $\sH$ is an ample invertible sheaf on $X$ and 
$\sO_X(-(K_X+D))\simeq\sH^{\otimes r}$. Therefore 
$(X, D)$ is a $2r$-dimensional log Fano manifold with the log Fano index $r$ 
for any $D\in|\sO(0; 1)|$. 
\end{example}

\begin{example}\label{zeroonebig}
Let 
\[
X:=\pr[\pr^r; 0^{r-1}, 1, m]
\]
with $m\geq 1$ and $D$ is an effective divisor on $X$ with $D\in|\sO(-m; 1)|$. 
If $m\geq 2$, then $D$ is unique in the linear system 
$|\sO(-m; 1)|$ and must be equal to the subbundle 
\[
X\pr[\pr^r; 0^{r-1}, 1]\cansub X=\pr[\pr^r; 0^{r-1}, 1, m].
\]
If $m=1$, then the dimension of the linear system is equal to $1$ 
by Lemma \ref{scroll_lem} \eqref{scroll_lem5}, 
and it is easy to show that 
any element is smooth and can be also seen the subbundle obtained by the above 
canonical projection after taking an automorphism of $X$ over $\pr^r$. 
In particular, any $D\in|\sO(-m; 1)|$ is a smooth divisor. 
Let $\sH:=\sO(1; 1)$. 
Then $\sH$ is an ample invertible sheaf on $X$ and 
$\sO_X(-(K_X+D))\simeq\sH^{\otimes r}$. Therefore 
$(X, D)$ is a $2r$-dimensional log Fano manifold with the log Fano index $r$ 
for any $D\in|\sO(-m; 1)|$. 
\end{example}

\begin{example}\label{zerozerobig}
Let 
\[
X:=\pr[\pr^r; 0^r, m]
\]
with $m\geq 2$ and $D$ is an effective divisor on $X$ such that 
$D\in|\sO(-m+1; 1)|$. 
We note that $D$ always has the support 
\[
D_0:=\pr[\pr^r; 0^r]\cansub X=\pr[\pr^r; 0^r, m]
\]
by Lemma \ref{scroll_lem} \eqref{scroll_lem6}. 
Furthermore, $D-D_0$ is the pull back of a hyperplane in $\pr^r$ by the projection 
$p\colon X\rightarrow\pr^r$. 
Therefore any $D\in|\sO(-m+1; 1)|$ 
is a simple normal crossing divisor. We also note that the dimension of the linear system 
$|\sO(-m+1; 1)|$ is equal to $r$ by Lemma \ref{scroll_lem} \eqref{scroll_lem5}. 
Let $\sH:=\sO(1; 1)$. Then $\sH$ is an ample invertible sheaf on $X$ 
and $\sO_X(-(K_X+D))\simeq\sH^{\otimes r}$. Therefore 
$(X, D)$ is a $2r$-dimensional log Fano manifold with the log Fano index $r$ 
for any $D\in|\sO(-m+1; 1)|$.
\end{example}

Now, we state some remarks about these examples. 

\begin{remark}\label{fanoQrmk}
In Example \ref{fanoQ}, the homomorphism 
\[
\tau^*\colon\Pic(X')\rightarrow\Pic(X)
\]
is an isomorphism. In particular, $\rho(X)=2$. 
\end{remark}

\begin{proof}
We can assume $\Bbbk=\C$. For the case $m=0$ is obvious 
since $X\simeq\pr^{r-1}\times\Q^{r+1}$. We consider the case $m>0$. 
Let $R\subset X$ be the ramification divisor of $\tau$. 
We know that the linear system $|\sO(0; 1)|$ in $X'$ gives 
a divisorial contraction morphism $f\colon X'\rightarrow Q$ contracting 
$E\simeq\pr^{r-1}\times\pr^r$ to $\pr^r$. 
We note that $B\subset X'$ is the pull back of some ample divisor $A\subset Q$. 
Thus 
$H_i((X'\setminus B)^{\text{an}}; \Z)=0$ for all $i>2r+r-2$ 
by \cite[p.25, (2.3) Theorem]{morse} 
for the proper morphism 
$f|_{X'\setminus B}\colon X'\setminus B\rightarrow Q\setminus A$ 
to an affine variety. Thus $H_c^i((X'\setminus B)^{\text{an}}; \Z)=0$ for all $i<r+2$ 
by Poincar\'e's duality. 
We know that there exists an exact sequence 
\[
H_c^2((X'\setminus B)^{\text{an}}; \Z)\rightarrow 
H^2((X')^{\text{an}}; \Z)\xrightarrow{\alpha}H^2(B^{\text{an}}; \Z)
\rightarrow H_c^3((X'\setminus B)^{\text{an}}; \Z).
\]
Thus $\alpha$ is an isomorphism. Applying the same argument to the composition 
$f\circ\tau\colon X\rightarrow Q$, we obtain the isomorphism 
$H^2(X^{\text{an}}; \Z)\xrightarrow{\sim}H^2(R^{\text{an}}; \Z)\simeq 
H^2(B^{\text{an}}; \Z)$. 
Therefore $H^2((X')^{\text{an}}; \Z)\simeq H^2(X^{\text{an}}; \Z)$. 
Therefore $\Pic(X')\simeq\Pic(X)$ by Lemma \ref{picard}. 
\end{proof}

\begin{remark}\label{rthreermk}
If $m<0$, then $(X, D)$ never be a log Fano manifold in Example \ref{rthree}.  
\end{remark}

\begin{proof}
Let $S:=\pr[\Q^r; m]\cansub X=\pr[\Q^r; 0^r, m]$, the section of the projection 
$p\colon X\rightarrow\Q^r$. 
Then $\sO_X(-(K_X+D))|_S\simeq\sO_{\Q^r}(r(m+1))$. 
Therefore $-(K_X+D)$ never be ample. 
\end{proof}

\begin{remark}\label{Tprmk}
If $m<1$, then $(X, D)$ never be 
a log Fano manifold in Example \ref{Tp}. 
\end{remark}

\begin{proof}
Let 
\[
S:=\pr_{\pr^r}(\sO(m))\cansub X=\pr_{\pr^r}(T_{\pr^r}\oplus\sO(m)).
\]
Then $\sO_X(-(K_X+D))|_S\simeq\sO_{\pr^r}(mr)$. 
Therefore $-(K_X+D)$ never be ample. 
\end{proof}

\begin{remark}\label{mukai1rmk}
For Examples \ref{burouappu}, \ref{pP}, \ref{kayaku}, \ref{fanoQ}, \ref{rthree}, 
\ref{rtwo}, \ref{Tp}, \ref{Pp}, \ref{zerozeroone}, \ref{zeroonebig}, 
\ref{zerozerobig}, two varieties $X_1$ and $X_2$ are not isomorphic to each other 
for any two $(X_1, D_1)$ and $(X_2, D_2)$ if either of the following holds: 
\begin{itemize}
\item
$(X_1, D_1)$ and $(X_2, D_2)$ are in different example boxes. 
\item
$(X_1, D_1)$ and $(X_2, D_2)$ are in a same example boxes but 
the parameters $m$ or $(m_1, m_2)$ are distinct. 
\end{itemize}
We also note that any $X$ in these examples is unique up to isomorphism 
except for Example \ref{fanoQ}. 
\end{remark}

We tabulate these examples in Section \ref{sbsc_mukai1}. 
We will show in Theorem \ref{mukai1} that these are the all examples of 
$2r$-dimensional log Fano manifolds $(X, D)$ wih the log Fano indices $r\geq 2$, 
$D\neq 0$ and $\rho(X)\geq 2$.

\begin{center}
\begin{table}\caption{The list of $2r$-dimensional log Fano manifolds $(X, D)$ 
with the log Fano indices $r\geq 2$, 
$\rho(X)\geq 2$ and nonzero boundaries}\label{maintable}
\begin{tabular}[t]{|c|c|c|} \hline
\rule[-1ex]{0ex}{3.5ex}
\small{No.} & $X$ & $D$  \\ \hline

\rule[-0.5ex]{0ex}{3ex}
\ref{burouappu} & $\Bl_{\pr^{r-2}}\pr^{2r}$  &  $\Bl_{\pr^{r-2}}\pr^{2r-1}$ with \\
\rule[-1ex]{0ex}{3ex}
& & $\pr^{r-2}\subset\pr^{2r-1}\subset\pr^{2r}$ linear \\ \hline

\rule[-0.5ex]{0ex}{3ex}
\ref{pP} & $\pr^{r-1}\times\pr^{r+1}$  &  $\pr^{r-1}\times\Q^r$  \\ \cline{3-3}
\rule[-0.5ex]{0ex}{3ex}
&  &  $\pr^{r-1}\times\pr^r\cup\pr^{r-1}\times\pr^r$   \\ \hline

\rule[-0.5ex]{0ex}{3ex}
\ref{kayaku} & $\pr[\pr^{r-1}; 0^r, m_1, m_2]$ & $\pr[\pr^{r-1}; 0^r, m_1]\cup
\pr[\pr^{r-1}; 0^r, m_2]$  \\  
\rule[-1ex]{0ex}{3ex}
& with $0\leq m_1\leq m_2, 1\leq m_2$ &    \\  \hline

\rule[-0.5ex]{0ex}{3ex}
\ref{fanoQ} & the double cover of & the strict transform of \\
\rule[-1ex]{0ex}{3ex}
& $\pr[\pr^{r-1}; 0^{r+1}, m]$ $(m\geq 0)$ with & 
$\pr[\pr^{r-1}; 0^{r+1}]$\,\,$(\simeq\pr^{r-1}\times\Q^r)$; \\ 
\rule[-1ex]{0ex}{3ex}
& the smooth branch $B\in|\sO(0; 2)|$ & smooth \\ \hline

\rule[-1.5ex]{0ex}{4ex}
\ref{rthree} & $(r\geq 3)$ $\pr[\Q^r; 0^r, m]$ with $m\geq 0$ & 
$\pr[\Q^r; 0^r](\simeq\pr^{r-1}\times\Q^r)$ \\ \hline

\rule[-0.5ex]{0ex}{3ex}
\ref{rtwo} & $(r=2)$ $\pr_{\pr^1\times\pr^1}(\sO^{\oplus 2}\oplus\sO(m_1, m_2))$ & 
$\pr_{\pr^1\times\pr^1}(\sO^{\oplus 2})(\simeq\pr^1\times\pr^1\times\pr^1)$ \\
\rule[-1ex]{0ex}{3ex}
& with $0\leq m_1\leq m_2$ & \\ \hline

\rule[-1.5ex]{0ex}{4ex}
\ref{Tp} & $\pr_{\pr^r}(T_{\pr^r}\oplus\sO(m))$ with $m\geq 1$  &  $\pr_{\pr^r}(T_{\pr^r})$\\ \hline

\rule[-0.5ex]{0ex}{3ex}
\ref{Pp} & $\pr^r\times\pr^r$ & $\pr_{\pr^r}(T_{\pr^r})\in|\sO(1, 1)|$\\ \cline{3-3}
\rule[-0.5ex]{0ex}{3ex}
& & $\pr^{r-1}\times\pr^r\cup\pr^r\times\pr^{r-1}$\\ \hline

\rule[-0.5ex]{0ex}{3ex}
\ref{zerozeroone} & $\pr[\pr^r; 0^r, 1]$ & 
$\pr[\pr^r; 0^{r-1}, 1](\simeq\Bl_{\pr^{r-2}}\pr^{2r-1})$\\ \cline{3-3}
\rule[-0.5ex]{0ex}{3ex}
& & $\pr[\pr^r; 0^r]\cup\pr[\pr^{r-1}; 0^r, 1]$\\ \hline

\rule[-1.5ex]{0ex}{4ex}
\ref{zeroonebig} & $\pr[\pr^r; 0^{r-1}, 1, m]$ with $m\geq 1$ & 
$\pr[\pr^r; 0^{r-1}, 1](\simeq\Bl_{\pr^{r-2}}\pr^{2r-1})$ \\ \hline

\rule[-1.5ex]{0ex}{4ex}
\ref{zerozerobig} & $\pr[\pr^r; 0^r, m]$ with $m\geq 2$ & 
$\pr[\pr^r; 0^r]\cup\pr[\pr^{r-1}; 0^r, m]$\\ \hline
\end{tabular}
\\
\begin{itemize}
\item[\ref{burouappu}]\,\,
$D$ is the strict transform of a hyperplane in $\pr^{2r}$ 
which contains the center.

\item[\ref{pP}]\,\,
$D\in|\sO(0, 2)|$. The upper column is  smooth and the lower is reducible.

\item[\ref{kayaku}]\,\,
Both components of $D$ to $X$ are obtained by the canonical projections. 

\item[\ref{fanoQ}]\,\,
$\pr[\pr^{r-1}; 0^{r+1}]\cansub\pr[\pr^{r-1}; 0^{r+1}, m]$. 
We request that both $B$ and $B\cap\pr[\pr^{r-1}; 0^{r+1}]$ are smooth.  

\item[\ref{rthree},] \ref{rtwo},\, \ref{Tp},\, \ref{zeroonebig}\,\,
The embeddings $D\subset X$ are obtained by the canonical projections. 

\item[\ref{Pp}]\,\,
$D\in|\sO(1, 1)|$. The upper column is smooth and 
the lower is non-smooth that is obtained by the union of the first and second 
pullbacks of hyperplanes. 

\item[\ref{zerozeroone}]\,\,
$\pr[\pr^r; 0^{r-1}, 1]$, $\pr[\pr^r; 0^r]\cansub\pr[\pr^r; 0^r, 1]$ and 
$\pr[\pr^{r-1}; 0^r, 1]$ is the pullback of a hyperplane in $\pr^r$. 

\item[\ref{zerozerobig}]\,\,
$\pr[\pr^r; 0^r]\cansub\pr[\pr^r; 0^r, m]$ and $\pr[\pr^{r-1}; 0^r, m]$ 
is the pullback of a hyperplane in $\pr^r$. 
\end{itemize}
\end{table}
\end{center}

%\medskip

\section{Theorems}\label{thm_section}

%\smallskip

In this section, we state the main propositions and theorems of classification results. 

%\smallskip

\subsection{Log Fano manifolds of del Pezzo type}\label{delpezzo_section}

%\smallskip

First, we give the classification result of $n$-dimensional log Fano manifolds $(X,D)$ 
with the log Fano indices $r\geq n-1$. The case $D=0$ is well-known as 
del Pezzo manifolds (see for example \cite[I \S 8]{fujitabook}), 
hyperquadric or projective space (see \cite{KO}). Hence we consider the case 
that $D\neq 0$. We note that the case $(n, r)=(2, 1)$ has been classified 
by Maeda \cite[\S 3]{Maeda}. We treat the log Fano pseudoindex 
instead of the log Fano index. 

\begin{proposition}\label{indn}
Let $(X,D)$ be an $n$-dimensional log Fano manifold 
with the log Fano pseudoindex $\iota$ and assume $D$ is nonzero. 
Then $\iota\leq n$. 
If $\iota=n$, then $X\simeq\pr^n$ and $D$ is a hyperplane section under this isomorphism. 
\end{proposition}

\begin{proof}
Assume that $\iota\geq n$. 
Choose an extremal ray $R$ with a minimal rational curve $[C]\in R$ 
such that $(D\cdot R)>0$. Then $l(R)\geq n+1$ by Lemma \ref{longray}. 
We know that $l(R)\leq n+1$, 
where the equality holds if and only if $X\simeq\pr^n$ and $D\in|\sO(1)|$ 
by \cite{CMSB}. 
\end{proof}

\begin{proposition}\label{dP}
Let $(X,D)$ be an $n$-dimensional log Fano manifold 
with the log Fano pseudoindex $\iota=n-1$ and assume $D$ is nonzero and $n\geq 3$. 
Then $X$ is isomorphic to $\pr^n$ or $\Q^n$ unless $n=3$ and $(X, D)$ is isomorphic 
to the case $\iota=2$ in Example \ref{2rminusone} $($cf. Section \ref{sbsc_mukai0}$)$. 
Moreover: 
\begin{itemize}
\item
If $X=\pr^n$, then $D\in|\sO(2)|$, i.e., $D$ is a smooth or reducible hyperquadric. 
\item
If $X=\Q^n$, then $D\in|\sO(1)|$, i.e., $D$ is a smooth hyperplane section. 
\end{itemize}
\end{proposition}

\begin{proof}
Let $R$ be an extremal ray with a minimal rational curve $[C]\in R$ 
such that $(D\cdot R)>0$. Then $l(R)\geq n$ holds by Lemma \ref{longray}. 
Let $\pi\colon X\rightarrow Y$ be the contraction morphism associated to $R$. 
By Wi\'sniewski's inequality (Theorem \ref{wisn_ineq}), 
\[
\dim\Exc(\pi)+\dim F\geq n+l(R)-1\geq 2n-1
\]
holds for any nontrivial fiber $F$ of $\pi$. Hence $\pi$ is of fiber type, 
that is, $X=\Exc(\pi)$ holds. 

If $\rho(X)=1$, then $X$ itself Fano manifold and the Fano pseudoindex of $X$ is 
larger than $n-1$ by Proposition \ref{rhoonefano}. 
If $\rho(X)=1$ and the Fano pseudoindex of $X$ is at least $n+1$, 
then $X\simeq\pr^n$ by \cite{CMSB} and $D\in|\sO(2)|$ holds. 
If $\rho(X)=1$ and the Fano pseudoindex is equal to $n$, then $(D\cdot C)=1$ 
by Lemma \ref{longray} \eqref{longray1}. 
Therefore $X\simeq\Q^n$ and $D\in|\sO(1)|$ by \cite{KO}. 

We consider the remaining case $\rho(X)\geq 2$. 
Then $\dim F=n-1$ for any fiber $F$ of $\pi$ and $l(R)=n$. 
Hence $(D\cdot C)=1$ holds by 
Lemma \ref{longray} \eqref{longray1}. 
Therefore $\pi$ is a $\pr^{n-1}$-bundle over a smooth projective curve $Y$ 
by \cite[Theorem 2]{fujita}. 
Then $Y\simeq\pr^1$ since any extremal ray of $X$ is 
spanned by a class of rational curve 
(this can also be shown by Theorem \ref{rat_conn}). Hence we can assume 
$X=\pr[\pr^1; a_0,\dots,a_{n-1}]$, where $0=a_0\leq a_1\leq\dots\leq a_{n-1}$. 
Thus $1\geq n-2$ by Corollary \ref{scroll_cor1}. 
Since $n\geq 3$, we have $n=3$, $a_1=0$ and 
$D\in|\sO(-a_2; 1)|$ by Corollary \ref{scroll_cor1} (1). 
That is exactly the case which we have considered in Example \ref{2rminusone} 
for the case $\iota=2$. 
\end{proof}

%\smallskip

\subsection{Log Fano manifolds related to Mukai conjecture}\label{mukaiconjsection}

%\smallskip

For a classification problem of Fano manifolds, it is very famous 
so called the Mukai conjecture \cite[Conjecture 4]{mukaiconj} such that 
the product of the Picard number and the 
Fano index minus one is less than or equal to the dimension for any Fano manifold. 
The case where the Fano index is larger than the half of the dimension is treated 
by Wi\'sniewski \cite{wisn90, wisn91}. 
We consider the log version, which are the main results in this article. 
%Let $(X, D)$ be a log Fano manifold 
%whose log Fano index is at least the half of the dimension of $X$ and $\rho(X)\geq 2$. 

\begin{thm}\label{mukai0}
Let $(X,D)$ be an $n$-dimensional log Fano manifold 
with the log Fano pseudoindex $\iota>n/2$, $D\neq 0$ and $\rho(X)\geq 2$. 
Then $n=2\iota-1$, and 
$(X, D)$ is isomorphic to the case 
in Example \ref{2rminusone} in Section \ref{ex_section}. 
\end{thm}

\begin{proof}
We prove the theorem by induction on $n$. The cases $n\leq 4$ have been done 
in Proposition \ref{dP}, thus we may assume $n\geq 5$. 

Choose an extremal ray $R$ with a minimal rational curve $[C]\in R$ 
as in Lemma \ref{longray} and let $\pi\colon X\rightarrow Y$ be 
the associated contraction. 
Then $l(R)\geq\iota+1\geq 4$. 
We also choose an irreducible component with the conductor divisor 
$(D_1, E_1)\subset D$ such that $(D_1\cdot R)>0$. 
We know that $\rho(D_1)\geq 2$ by Lemma \ref{rhoone} \eqref{rhoone2}, and 
$D$ is an $(n-1)$-dimensional snc Fano variety 
whose snc Fano pseudoindex $\geq\iota$ 
by Theorem \ref{fujinothm} \eqref{fujinothm1}. Thus $(D_1, E_1)$ is an 
$(n-1)$-dimensional log Fano manifold whose log Fano pseudoindex $\geq\iota$, 
$\rho(D_1)\geq 2$ and 
$\iota>n/2>(n-1)/2$. 
Hence $E_1=0$ (hence $D=D_1$) by induction step. 
Applying \cite[Corollary 4.3]{Occ} to $D$, 
we have $n-1=2(\iota-1)$ and $D\simeq\pr^{\iota-1}\times\pr^{\iota-1}$. 

We know in Lemma \ref{rhoone} \eqref{rhoone2} that $\pi|_D$ contracts a curve. 
Since $D\simeq\pr^{\iota-1}\times\pr^{\iota-1}$, 
$\pi|_D\colon D\rightarrow\pi(D)$ is not birational. Thus 
$\pi\colon X\rightarrow Y$ is of fiber type by Lemma \ref{rhoone} \eqref{rhoone1}, 
and $\pi|_D\colon D\rightarrow Y$ is surjective since $(D\cdot R)>0$. 
We know that $\pi|_D\colon D\rightarrow Y$ is an algebraic fiber space 
by Lemma \ref{longray}. 
Hence $\pi|_D$ is isomorphic to the first projection 
\[
p_1\colon\pr^{\iota-1}\times\pr^{\iota-1}\rightarrow\pr^{\iota-1}.
\]
In particular, $\dim(\pi^{-1}(y)\cap D)=\iota-1$ for any closed point 
$y\in Y\simeq\pr^{\iota-1}$. 
Therefore, $\pi\colon X\rightarrow Y$ is a 
$\pr^\iota$-bundle and $\pi|_D\colon D\rightarrow Y$ is a 
$\pr^{\iota-1}$-subbundle by Proposition \ref{structure_cor} \eqref{structure_cor1}. 

Since $D\simeq\pr^{\iota-1}\times\pr^{\iota-1}$, there exists an integer $m\in\Z$ such that 
$(\pi|_D)_*\sN_{D/X}\simeq\sO_{\pr^{\iota-1}}(-m)^{\oplus\iota}$ by Lemma \ref{lemP}. 
We also know by Lemma \ref{lemP} that $X\simeq\pr_{\pr^{\iota-1}}(\pi_*\sO_X(D))$ 
and the embedding $D\subset X$ is obtained by the surjection $\alpha$ 
in the natural exact sequence 
\[
0\rightarrow\sO_{\pr^{\iota-1}}\rightarrow\pi_*\sO_X(D)\xrightarrow{\alpha}
(\pi|_D)_*\sN_{D/X}\rightarrow 0.
\]
Since $\iota-1=(n+1)/2-1\geq 2$, this exact sequence always splits. Hence we have 
$\pi_*\sO_X(D)\simeq\sO_{\pr^{\iota-1}}\oplus\sO_{\pr^{\iota-1}}(-m)^{\oplus\iota}$ and 
$D\subset X$ is 
obtained by the canonical projection 
\[
\sO_{\pr^{\iota-1}}\oplus\sO_{\pr^{\iota-1}}(-m)^{\oplus\iota}\rightarrow
\sO_{\pr^{\iota-1}}(-m)^{\oplus\iota}.
\] 
This case has been already considered in Corollary \ref{scroll_cor1} \eqref{scroll_cor11}; 
$m\geq 0$ holds. This is exactly the case which we have been considered in 
Example \ref{2rminusone}. Therefore we have completed the proof of 
Theorem \ref{mukai0}.
\end{proof}

We recall Wi\'sniewski's classification result. 

\begin{thm}[{\cite{wisn91}}]\label{wisn_classify}
Let $X$ be an $n$-dimensional Fano manifold with the Fano index $r$. 
If $n=2r-1$ and $\rho(X)\geq 2$, then $X$ is isomorphic to one of the following: 
\begin{enumerate}
\renewcommand{\theenumi}{\roman{enumi}}
\renewcommand{\labelenumi}{\rm{\theenumi)}}
\item\label{wisn_classify1}
$\pr^{r-1}\times\Q^r,$
\item\label{wisn_classify1}
$\pr_{\pr^r}(T_{\pr^r}),$
\item\label{wisn_classify1}
$\pr[\pr^r; 0^{r-1}, 1].$
\end{enumerate}
\end{thm}

Using Theorems \ref{mukai0}, \ref{wisn_classify}, 
we classify $(X, D)$ a $2r$-dimensional log Fano manifold 
with the log Fano index $r\geq 2$ 
and $D\neq 0$. We note that the case $r=1$ has been classified by Maeda 
\cite[\S 3]{Maeda}.

\begin{thm}[Main Theorem]\label{mukai1}
Let $(X, D)$ be a $2r$-dimensional log Fano manifold with the log Fano index $r\geq 2$,  $D\neq 0$ and $\rho(X)\geq 2$. 
Then the possibility of $X$ and $D$ is exactly  in the Examples 
\ref{burouappu}, \ref{pP}, \ref{kayaku}, \ref{fanoQ}, \ref{rthree}, 
\ref{rtwo}, \ref{Tp}, \ref{Pp}, \ref{zerozeroone}, \ref{zeroonebig}, 
\ref{zerozerobig} 
(See Table \ref{maintable} for the list of $X$ and $D$).
\end{thm}

We prove Theorem \ref{mukai1} in Section \ref{prf_section}.

%\smallskip

\subsection{Classification of Mukai-type log Fano manifolds}\label{mukaisection}

%\smallskip

\begin{corollary}\label{coindex3}
We have classified $n$-dimensional log Fano manifolds $(X, D)$ with the 
log Fano indices $r\geq n-2$.
\end{corollary}

\begin{proof}
We can assume $r=n-2$ since the cases $r\geq n-1$ have been treated 
in Section \ref{delpezzo_section}. 

The case $D=0$ is well-known, called as Mukai manifolds 
(see \cite{isk77,MoMu,mukai,wis,wisn90,wisn91}). 

The case $n=3$, $D\neq 0$ has been classified by Maeda \cite{Maeda}. 

The case $n\geq 4$, $D\neq 0$, $\rho(X)\geq 2$ is already known by 
Propositions \ref{indn}, \ref{dP} and Theorems \ref{mukai0}, \ref{mukai1}. 

The remaining case $n\geq 4$, $D\neq 0$, $\rho(X)=1$ is evident from 
Proposition \ref{rhoonefano}; $X$ and $D$ is isomorphic to one of the following: 
\begin{enumerate}
\renewcommand{\theenumi}{\Roman{enumi}}
\renewcommand{\labelenumi}{\rm{(\theenumi)}}
\item\label{MI}
$X\simeq\pr^n$ and $D\in|\sO(3)|$. 
\item\label{MII}
$X\simeq\Q^n$ and $D\in|\sO(2)|$.
\item\label{MIII}
$X\simeq V_d$ and $D\in|\sO(1)|$ with $1\leq d\leq 5$, 
where $V_d$ is a del Pezzo manifold of degree $d$ in the sense of Takao Fujita 
\cite[Theorem 8.11, 1)--5)]{fujitabook}, and $\sO(1)$ is the ample generator of 
$\Pic(V_d)$. 
\end{enumerate}

Conversely, we know that general elements in the linear systems of 
\eqref{MI}--\eqref{MIII} are smooth. Hence the cases \eqref{MI}--\eqref{MIII} 
actually occur.
\end{proof}

%\medskip

\section{Proof of Main Theorem \ref{mukai1}}\label{prf_section}

In this section, we give a proof of Main Theorem \ref{mukai1}.  

Let $L$ be an ample divisor on $X$ such that $-(K_X+D)\sim rL$. 
Pick an extremal ray $R$ with a minimal rational curve $[C]\in R$ 
such that $(D\cdot R)>0$ and let $\pi\colon X\rightarrow Y$ 
be the associated contraction morphism. 
Then we know that $l(R)\geq r+1\geq 3$ by Lemma \ref{longray}. 
We note that $D$ is a $(2r-1)$-dimensional snc Fano variety 
whose snc Fano index is divisible by $r$ 
by Theorem \ref{fujinothm} \eqref{fujinothm1}. 
Let $(D_1, E_1)\subset D$ be an irreducible component of $D$ with the conductor divisor 
such that $(D_1\cdot R)>0$. 
By the assumption $\rho(X)\geq 2$ and 
Lemma \ref{rhoone} \eqref{rhoone2}, the morphism 
$\pi|_{D_1}\colon D_1\rightarrow\pi(D_1)$ is not a finite morphism and 
$\rho(D_1)\geq 2$ holds. 
Since $(D_1, E_1)$ is a $(2r-1)$-dimensional log Fano manifold whose log Fano index is 
divisible by $r$, 
the possibility of the morphism $\pi|_{D_1}\colon D_1\rightarrow\pi(D_1)$ 
(which is an algebraic fiber space by Lemma \ref{longray}) is isomorphic to exactly 
one of the following list 
by Theorems \ref{mukai0} and \ref{wisn_classify}.
\begin{enumerate}
\renewcommand{\theenumi}{\arabic{enumi}}
\renewcommand{\labelenumi}{$(\theenumi)$}
\item\label{main1}
$\pr^{r-1}\times\Q^r\xrightarrow{p_1}\pr^{r-1}$, 
where $E_1=0$.
\item\label{main2}
$\pr[\pr^{r-1}; 0^r, m]\xrightarrow{p}\pr^{r-1}$, where $E_1\in|\sO(-m; 1)|$ 
with $m\geq 0$.
\item\label{main3}
$\pr^{r-1}\times\Q^r\xrightarrow{p_2}\Q^r$, 
where $E_1=0$.
\item\label{main4}
$\pr_{\pr^r}(T_{\pr^r})\xrightarrow{p}\pr^r$, 
where $E_1=0$.
\item\label{main5}
$\pr[\pr^{r}; 0^{r-1}, 1]\xrightarrow{p}\pr^r$, 
where $E_1=0$.
\item\label{main6}
$\pr^r\times\pr^{r-1}\xrightarrow{p_1}\pr^r$, 
where $E_1\in|\sO_{\pr^r\times\pr^{r-1}}(1, 0)|$ (the case in Theorem \ref{mukai0} with $m=0$).
\item\label{main7}
$\pr[\pr^{r-1}; 0^r, m]\xrightarrow{\phi}Z$, 
the divisorial contraction morphism contracting $E_1\simeq\pr^{r-1}\times\pr^{r-1}$ 
to $\pr^{r-1}$, where $m>0$. 
\item\label{main8}
$\Bl_{\pr^{r-2}}\pr^{2r-1}\xrightarrow{\Bl}\pr^{2r-1}$, the blowing up of $\pr^{2r-1}$ 
along a linear subspace $\pr^{r-2}$, 
where $E_1=0$. 
\end{enumerate}

\begin{remark}
For the cases \eqref{main1}, \eqref{main2} and \eqref{main8}, 
$\dim(F\cap D_1)=r$ for any nontrivial fiber $F$ of $\pi$. 
For the cases \eqref{main3}, \eqref{main4}, \eqref{main5}, \eqref{main6} and 
\eqref{main7}, $\dim(F\cap D_1)=r-1$ for any nontrivial fiber $F$ of $\pi$. 
\end{remark}

We separate the cases whether the contraction morphism $\pi$ 
is of fiber type (Section \ref{fiber_type}) or not (Section \ref{birational_type}).

\subsection{Fiber type case}\label{fiber_type}
Here, we consider the case where $\pi$ is of fiber type. 
Since $\dim F\geq l(R)-1\geq r\geq 2$ for any fiber $F$ of $\pi$, we have 
$\dim D_1>\dim Y$. Hence $\pi|_{D_1}$ is surjective and belongs to the cases 
\eqref{main1}--\eqref{main6} (we note that $\pi|_{D_1}$ is an algebraic fiber space by 
Lemma \ref{longray}).

\noindent\textbf{The cases \eqref{main1} and \eqref{main2}}

First, we consider the cases \eqref{main1} and \eqref{main2}. 
Then $\dim(\pi|_{D_1})^{-1}(y)=r$ for 
any closed point $y\in Y\simeq\pr^{r-1}$. Thus one of 
\eqref{structure_cor2b}, \eqref{structure_cor2c} 
or \eqref{structure_cor2d} in Proposition \ref{structure_cor} holds.

\underline{The case \eqref{main1}}

Since $\pi|_D$ is a $\sQ^r$-bundle over $Y\simeq\pr^{r-1}$, 
only the case \eqref{structure_cor2b} or \eqref{structure_cor2d} can occurs. 

First, we consider the case \eqref{structure_cor2b}. 
Since $D\simeq\pr^{r-1}\times\Q^r$, 
$\pi|_D$ is isomorphic to the first projection and 
we can write $\sN_{D/X}\simeq\sO_{\pr^{r-1}\times\Q^r}(-m,1)$ for some 
integer $m\in\Z$. Then 
$(\pi|_D)_*\sN_{D/X}\simeq\sO_{\pr^{r-1}}(-m)^{\oplus r+2}$ by Lemma \ref{PQlem}, 
and the sequence 
\[
0\rightarrow\sO_{\pr^{r-1}}\rightarrow\pi_*\sO_X(D)\xrightarrow{\alpha}
\sO_{\pr^{r-1}}(-m)^{\oplus r+2}\rightarrow 0
\]
is exact. 
Furthermore, $X$ is obtained as a smooth divisor belonging to 
$|p^*\sO_{\pr^{r-1}}(s)\otimes\sO_{\pr}(2)|$ in $\pr:=\pr_{\pr^{r-1}}(\pi_*\sO_X(D))$ 
for some $s\in\Z$, where $p\colon\pr\rightarrow\pr^{r-1}$ is the projection, 
$D$ is the complete intersection of $X$ with 
$H:=\pr[\pr^{r-1}; (-m)^{r+2}]$ in $\pr$. Here 
$H\subset\pr$ is the subbundle of $p$ 
obtained by the surjection $\alpha$ in the above exact sequence, by Lemma \ref{lemQ}. 
Under the isomorphism $H\simeq\pr^{r-1}\times\pr^{r+1}$, the divisor 
$D\simeq\pr^{r-1}\times\Q^r$ belongs to $|\sO_{\pr^{r-1}\times\pr^{r+1}}(s-2m, 2)|$. 
Thus $s\geq 2m$ since 
$h^0(\pr^{r-1}\times\pr^{r+1}, \sO_{\pr^{r-1}\times\pr^{r+1}}(t, 2))=0$ for any $t<0$. 
If $s>2m$, then the restriction homomorphism 
$\Pic(\pr^{r-1}\times\pr^{r+1})\rightarrow\Pic(D)$ is isomorphism by 
Lefschetz hyperplane theorem and 
$\sO_D(-K_D)\simeq\sO_{\pr^{r-1}\times\pr^{r+1}}(r-(s-2m), r)|_D$ by the 
adjunction theorem, but we know that $-K_D$ is divisible by $r$, 
which leads to a contradiction. 
Therefore we have $s=2m$. 

\begin{claim}\label{mclaim}
$m\geq 0$ holds. 
\end{claim}

\begin{proof}
We first consider the case $r=2$. Since $\rho(X)=2$, we can write $\NE(X)=R+R'$ 
and let the contraction morphism associated to $R'$ be $\pi'\colon X\rightarrow Y'$. 
We note that any nontrivial fiber $F'$ of $\pi'$ satisfies $\dim F'=1$ since 
any curve in $F'$ never  be contracted by $\pi$. 
If $(D\cdot R')>0$, then any nontrivial fiber $F'$ of $\pi'$ satisfies $\dim F'\geq 2$ 
by the same argument in Proposition \ref{structure_cor}, this is a contradiction. 
If $(D\cdot R')<0$, then $\Exc(\pi')\subset D$. Hence $m>0$. 
If $(D\cdot R')=0$, then $R'$ is a $K_X$-negative extremal ray 
and $l(R')\geq 2$ by the same argument in Proposition \ref{structure_cor}. 
Hence $\pi'$ is of fiber type by Wi\'sniewski's inequality (Theorem \ref{wisn_ineq}). 
Thus $\pi'|_D$ is not a finite morphism since $(D\cdot R')=0$. Therefore $m\geq 0$. 

Now we consider the case $r\geq 3$. The above exact sequence always splits, 
hence $H=\pr[\pr^{r-1}; (-m)^{r+2}]\cansub\pr=\pr[\pr^{r-1}; (-m)^{r+2}, 0]$. 
Assume that $m<0$ holds. 
The total coordinate ring of $\pr$ is the $\Z^{\oplus 2}$-graded polynomial ring 
\[
\Bbbk[x_0, \dots, x_{r-1}, y_0, y_1, \dots, y_{r+2}]
\]
with the grading
\begin{eqnarray*}
\deg x_i & = & (1, 0) \ (1\leq i\leq r-1),\\
\deg y_0 & = & (0, 1),\\
\deg y_i & = & (m, 1) \ (1\leq i\leq r+2).
\end{eqnarray*}
$X$ is obtained by a graded equation of bidegree $(2m, 1)$. 
Since $m<0$, any bidegree $(2m, 1)$ polynomial is obtained by linear combinations of 
the elements in $\{y_iy_j\}_{1\leq i\leq j\leq r+2}$. Then any divisor 
obtained by graded equations with the grade $(2m, 1)$ have singular points 
along the points defined by the graded equations $y_1=\dots=y_{r+2}=0$ 
by the Jacobian criterion. 
This is a contradiction since $X$ must be a smooth divisor. 
Therefore $m\geq 0$. 
\end{proof}

Hence the above exact sequence splits. We now normalize the bundle structures 
for simplicity. That is, we rewrite 
\[
H:=\pr[\pr^{r-1}; 0^{r+2}]\cansub\pr:=\pr[\pr^{r-1}; 0^{r+2}, m]
\]
with $m\geq 0$, 
$X$ is a smooth divisor on $\pr$ with 
$X\in|\sO(0; 2)|$ and $D=X\cap H$ and $D$ is smooth. 
Since $H\simeq\pr^{r-1}\times\pr^{r+1}$ and $D\simeq\pr^{r-1}\times\Q^r$, 
we can take the pull back of a point $S(\simeq\pr^{r-1})\subset H\simeq\pr^{r-1}\times\pr^{r+1}\xrightarrow{p_2}\pr^{r+1}$ 
in $\pr^{r+1}$ such that $S\cap D=\emptyset$. 
We can assume that $S$ is the section of $p\colon\pr\rightarrow\pr^{r-1}$ 
obtained by the canonical first projection, that is, 
\[
S=\pr[\pr^{r-1}; 0]\cansub\pr=\pr[\pr^{r-1}; 0^{r+2}, m].
\]
Then the relative linear projection from $S$ over $\pr^{r-1}\simeq Y$ 
obtains the morphism 
\[
\sigma\colon\pr\setminus S\rightarrow 
X':=\pr[\pr^{r-1}; 0^{r+1}, m]
\]
over $\pr^{r-1}\simeq Y$. 
The restriction of $\sigma$ to $X$ gives a double cover morphism 
$\tau\colon X\rightarrow X'$.
It is easy to show that the branch divisor $B\subset X'$ of $\tau$ 
is the smooth divisor on $X'$ such that $B\in|\sO(0; 2)|$. 
Since the strict transform of the divisor 
\[
D':=\pr[\pr^{r-1}; 0^{r+1}]\cansub X'=\pr[\pr^{r-1}; 0^{r+1}, m]
\]
in $X'$ is exactly $D$, 
the intersection $B\cap D'$ is also smooth. 
This is exactly the case in Example \ref{fanoQ}. 

Now, we consider the case \eqref{structure_cor2d}. We write $\sE:=\pi_*\sO_X(L)$, then 
\[
X\simeq\pr_{\pr^{r-1}}(\sE)\xrightarrow{p}\pr^{r-1}.
\]
We can write $\sO_{\pr}(1)|_D\simeq\sO_{\pr^{r-1}\times\Q^r}(-m, 1)$ for some 
integer $m\in\Z$, where $\sO_\pr(1)$ is the tautological invertible sheaf on $X$ 
with respect to the projection $p$. Hence 
$\sE\simeq(p|_D)_*(\sO_{\pr}(1)|_D)\simeq\sO_{\pr^{r-1}}(-m)^{\oplus r+2}$ holds 
by Lemma \ref{PQlem}. Therefore $X\simeq\pr^{r-1}\times\pr^{r+1}$, 
which is exactly the case in Example \ref{pP}.

\underline{The case \eqref{main2}}

For convenience, let $m_1:=m$, where $m$ is in \eqref{main2}.
Then only the case \eqref{structure_cor2c} can occurs 
since $\pi|_{D_1}$ is a $\pr^r$-bundle over $Y\simeq\pr^{r-1}$. 
We note that $D$ has exactly two irreducible components $D_1$ and $D_2$ 
since $E_1$ is irreducible. 
We note that $(D_2\cdot R)>0$ since $\pi$ is of fiber type 
and $\pi|_{E_1}$ is surjective. Hence $\rho(D_2)\geq 2$ by the 
previous argument. Therefore $\pi|_{D_2}\colon D_2\rightarrow Y$ is isomorphic to 
\[
\pr[\pr^{r-1}; 0^r, m_2]\xrightarrow{p}\pr^{r-1}
\]
with $m_2\geq 0$. That is, $D_2$ also satisfy 
the case \eqref{main2} by repeating the same argument. 
We can assume $0\leq m_1\leq m_2$. 
Under the isomorphism 
$D_1\simeq\pr[\pr^{r-1}; 0^r, m_1]$, 
we can write $\sN_{D_1/X}\simeq\sO(u; 1)$ with 
$u\in\Z$. 
We have $u=-m_2$ since $\sN_{D_1/X}|_{D_1\cap D_2}\simeq\sN_{D_1\cap D_2/D_2}$ 
and $\sN_{D_1\cap D_2/D_2}\simeq\sO_{\pr^{r-1}\times\pr^{r-1}}(-m_2, 1)$. 
Hence 
\begin{eqnarray*}
p_*\sN_{D_1/X} & \simeq &  
p_*(p^*\sO_{\pr^{r-1}}(-m_2)\otimes\sO_{\pr}(1))\\
 & \simeq & \sO_{\pr^{r-1}}(-m_2)^{\oplus r}\oplus\sO_{\pr^{r-1}}(m_1-m_2).
\end{eqnarray*}
Thus the exact sequence 
\[
0\rightarrow\sO_{\pr^{r-1}}\rightarrow\pi_*\sO_X(D_1)\rightarrow
\sO_{\pr^{r-1}}(-m_2)^{\oplus r}\oplus\sO_{\pr^{r-1}}(m_1-m_2)\rightarrow 0
\]
splits since $m_1\leq m_2$. Therefore 
\[
X\simeq\pr[\pr^{r-1}; 0^r, m_1, m_2]
\]
with $0\leq m_1\leq m_2$. 
We note that $D\in|\sO(-m_1-m_2; 2)|$ by Corollary \ref{scroll_cor2}. 
This is exactly the case in Example \ref{kayaku}.

\noindent\textbf{The cases \eqref{main3}--\eqref{main6}}

Next, we consider the cases \eqref{main3}--\eqref{main6}. 
Then $\dim(\pi|_{D_1})^{-1}(y)=r-1$ 
for any closed point $y\in Y$. Hence 
only the case \eqref{structure_cor1} in Proposition \ref{structure_cor} occurs.

\underline{The case \eqref{main3}}

In this case, $Y$ is isomorphic to $\Q^r$. 

First, we consider the case $r=2$. 
$\pi|_D$ is isomorphic to  $p_{23}\colon\pr^1\times\pr^1\times\pr^1\rightarrow\pr^1\times\pr^1$ 
and we can write $\sN_{D/X}\simeq\sO_{\pr^1\times\pr^1\times\pr^1}(1, -m_1, -m_2)$ 
with $m_1$, $m_2\in\Z$. 

\begin{claim}
$m_1$, $m_2\geq 0$ holds. 
\end{claim}

\begin{proof}
It is enough to show $m_1\geq 0$. 
Let $f=\{t\}\times\pr^1\subset \pr^1\times\pr^1\simeq Y$ be an arbitrary fiber 
of $p_1:\pr^1\times\pr^1\rightarrow\pr^1$, where $t\in\pr^1$. 
Let $X_f$ (resp.\ $D_f$) be the intersection of $\pi^{-1}(f)$ and $X$ (resp.\ $D$). 
Then $X_f\rightarrow f$ is a $\pr^2$-bundle, 
$D_f$ is a smooth divisor in $X_f$ with $D_f\neq 0$ and 
\[
\sO_{X_f}(-(K_{X_f}+D_f))\simeq\sO_X(-(K_X+D))|_{X_f}\simeq\sO_X(2L)|_{X_f}.
\]
Thus $(X_f, D_f)$ is a $3$-dimensional log Fano manifold whose log Fano index is an even number and $\rho(X_f)=2$. Hence 
\[
D_f=\pr[\pr^1; 0^2](\simeq\pr^1\times\pr^1)
\cansub X_f=\pr[\pr^1; 0^2, m]
\]
with $m\geq 0$ by Proposition \ref{dP}, thus 
$\sN_{D_f/X_f}\simeq\sO_{\pr^1\times\pr^1}(1, -m)$. 
Since $\sN_{D_f/X_f}\simeq\sN_{D/X}|_{D_f}\simeq\sO_{\pr^1\times\pr^1}(1, -m_1)$, 
we have $m_1=m\geq 0$. 
\end{proof}

We know that ${p_{23}}_*\sN_{D/X}\simeq\sO_{\pr^1\times\pr^1}(-m_1, -m_2)^{\oplus 2}$ 
by Lemma \ref{PQlem}. Hence we can show that the exact sequence 
obtained by Lemma \ref{lemP}
\[
0\rightarrow\sO_{\pr^1\times\pr^1}\rightarrow\pi_*\sO_X(D)\rightarrow
\sO_{\pr^1\times\pr^1}(-m_1, -m_2)^{\oplus 2}\rightarrow 0
\]
splits. Hence we we  can show that 
\[
D=\pr_{\pr^1\times\pr^1}(\sO_{\pr^1\times\pr^1}^{\oplus 2})\cansub
X=\pr_{\pr^1\times\pr^1}(\sO_{\pr^1\times\pr^1}^{\oplus 2}\oplus\sO_{\pr^1\times\pr^1}(m_1, m_2))
\]
with $0\leq m_1\leq m_2$ 
by Lemma \ref{lemP}. 
This is exactly the case in Example \ref{rtwo}.

We now consider the remaining case $r\geq 3$. 
We can write the normal sheaf $\sN_{D/X}\simeq\sO_{\pr^{r-1}\times\Q^r}(1, -m)$ with $m\in\Z$. 
Then $(\pi|_D)_*\sN_{D/X}\simeq\sO_{\Q^r}(-m)^{\oplus r}$ 
by Lemma \ref{PQlem}. 
Hence we can see that the exact sequence 
obtained by Lemma \ref{lemP}
\[
0\rightarrow\sO_{\Q^r}\rightarrow\pi_*\sO_X(D)
\rightarrow\sO_{\Q^r}(-m)^{\oplus r}\rightarrow 0
\]
splits. Hence we can show that 
\[
D=\pr[\Q^r; 0^r]\cansub X=\pr[\Q^r; 0^r, m]
\]
by Lemma \ref{lemP}. This is exactly the case 
in Example \ref{rthree}; the divisor $-(K_X+D)$ is ample 
if and only if $m\geq 0$ by Remark \ref{rthreermk}.

\underline{The case \eqref{main4}}

We can write $(\pi|_D)_*\sN_{D/X}\simeq T_{\pr^r}\otimes\sO_{\pr^r}(-m)$ 
with $m\in\Z$ 
by Lemma \ref{lemP}. Hence we obtain the exact sequence such that the 
right hand of the sequence has been seen 
in Lemma \ref{lemP}:
\[
0\rightarrow\sO_{\pr^r}\rightarrow\pi_*\sO_X(D)\rightarrow T_{\pr^r}\otimes\sO_{\pr^r}(-m)\rightarrow 0.
\]
It is well known that 
\[
\Ext_{\pr^r}^1(T_{\pr^r}\otimes\sO_{\pr^r}(-m), \sO_{\pr^r})\simeq
\begin{cases}
0 & (m\neq 0)\\
\Bbbk & (m=0).
\end{cases}
\]
We also know that all unsplit exact sequences for the case $m=0$ are obtained by 
the canonical exact sequences 
\[
0\rightarrow\sO_{\pr^r}\rightarrow\sO_{\pr^r}(1)^{\oplus r+1}
\rightarrow T_{\pr^r}\rightarrow 0.
\]
If the exact sequence is not split, then $X\simeq\pr^r\times\pr^r$ 
by the above argument. This case has been considered by Example \ref{Pp}. 
If the exact sequence splits, then we can show that
\[
D=\pr_{\pr^r}(T_{\pr^r})\cansub X=\pr_{\pr^r}(T_{\pr^r}\oplus\sO_{\pr^r}(m)).
\] 
This case has been considered 
by Example \ref{Tp}; the divisor  $-(K_X+D)$ is ample 
if and only if $m\geq 1$ by Remark \ref{Tprmk}.

\underline{The case \eqref{main5}}

We can write $(\pi|_D)_*\sN_{D/X}\simeq(\sO_{\pr^r}^{\oplus r-1}\oplus\sO_{\pr^r}(1))
\otimes\sO_{\pr^r}(-m)$ with $m\in\Z$ by Lemma \ref{lemP}. Since $r\geq 2$, 
the exact sequence 
\[
0\rightarrow\sO_{\pr^r}\rightarrow\pi_*\sO_X(D)\rightarrow
(\sO_{\pr^r}^{\oplus r-1}\oplus\sO_{\pr^r}(1))
\otimes\sO_{\pr^r}(-m)\rightarrow 0
\]
splits. Thus 
\[
X\simeq\pr[\pr^r; 0^{r-1}, 1, m]
\]
and $D\in|\sO(-m; 1)|$. 
Since $\sO_X(-K_X)\simeq\sO(r-m; r+1)$, 
we have $\sO_X(L)\simeq\sO(1; 1)$. 
We know in Corollary \ref{scroll_cor1} \eqref{scroll_cor12} such that $m\geq 0$; 
this case has been considered in Examples \ref{zerozeroone} and \ref{zeroonebig}.

\underline{The case \eqref{main6}}

We can write $(\pi|_{D_1})_*\sN_{D_1/X}\simeq\sO_{\pr^r}(-m)^{\oplus r}$ with $m\in\Z$ 
by Lemma \ref{lemP}. Since $r\geq 2$, 
the exact sequence 
\[
0\rightarrow\sO_{\pr^r}\rightarrow\pi_*\sO_X(D)\rightarrow
\sO_{\pr^r}(-m)^{\oplus r}\rightarrow 0
\]
splits. Thus 
\[
X\simeq\pr[\pr^r; 0^r, m].
\]
We know in Corollary \ref{scroll_cor1} \eqref{scroll_cor12} such that $m\geq 0$ holds; 
this case has been considered in Examples \ref{Pp}, \ref{zerozeroone} and \ref{zerozerobig}.

\subsection{Birational type case}\label{birational_type}

We know that $\pi|_{D_1}\colon D_1\rightarrow \pi(D_1)$ is a birational morphism by 
Lemma \ref{rhoone} \eqref{rhoone1} 
and an algebraic fiber space by Lemma \ref{longray}. 
Hence $\pi|_{D_1}\colon D_1\rightarrow \pi(D_1)$ belongs to 
the cases \eqref{main7} and \eqref{main8}. 
However, we have $\dim(D_1\cap F)=r-1$ for any nontrivial fiber $F$ of $\pi$ 
for the case \eqref{main7}; this contradicts to 
Proposition \ref{structure_cor} \eqref{structure_cor1}. 
For the case \eqref{main8}, we have $\dim(D\cap F)=r$ for any nontrivial fiber 
$F$ of $\pi$. 
Thus only the case \eqref{structure_cor2a} in Proposition \ref{structure_cor} occurs. 
That is, $Y$ is smooth and $\pi$ is the blowing up 
along a smooth projective subvariety $W\subset Y$ of 
dimension $r-2$. Let $D_Y:=\pi(D)\subset Y$. 
Then $D_Y\simeq\pr^{2r-1}$, and $W\subset D_Y$ 
is a linear subspace of dimension $r-2$ under the isomorphism $D_Y\simeq\pr^{2r-1}$. 
Let $E\subset X$ be the exceptional divisor of $\pi$. 
Then $\pi^*D_Y=D+E$. 
We note that there exists a divisor $L_Y$ on $Y$ such that 
$\pi^*\sO_Y(L_Y)\simeq\sO_X(L+E)$ by Theorem \ref{cone} 
since $(E\cdot C)=-1$ and $(L\cdot C)=1$. 
Therefore
\[
\sO_Y(rL_Y)\simeq\sO_Y(-(K_Y+D_Y))
\]
by Theorem \ref{cone} since $\pi^*\sO_Y(rL_Y)\simeq\sO_X(rL+rE)
\simeq\sO_X(-(K_X+D)+rE)\simeq\sO_X(-\pi^*K_Y-D-E)\simeq\pi^*\sO_Y(-(K_Y+D_Y))$. 

\begin{claim}\label{cont_ample}
$(Y, D_Y)$ is also a $2r$-dimensional log Fano manifold whose log Fano index is 
divisible by $r$. 
\end{claim}

\begin{proof}
It is enough to show that $L_Y$ is an ample divisor on $Y$. 
We know that $\NE(Y)$ is a closed convex cone since so is $\NE(X)$. 
Hence it is enough to show that $(L_Y\cdot C_Y)>0$ 
for any irreducible curve $C_Y\subset Y$. 
If $C_Y\not\subset W$, taking the strict transform $\widehat{C}_Y$ of $C_Y$ in $X$, 
then $(L_Y\cdot C_Y)=(L\cdot\widehat{C}_Y)+(E\cdot\widehat{C}_Y)>0$. 
Hence it is enough to show the case $C_Y\subset W$. 
We note that $W\subset D_Y$ and all curves in $D_Y$ are numerically proportional 
since $D_Y\simeq\pr^{2r-1}$. Therefore we can reduce to 
the case $C_Y\not\subset W$.
\end{proof}

Since $D_Y\simeq\pr^{2r-1}$, we have $\rho(Y)=1$ by Lemma \ref{rhoone} \eqref{rhoone2}. 
Therefore $Y\simeq\pr^{2r}$ and $D_Y$ is a hyperplane under this isomorphism 
by \cite[Theorem 7.18]{fujitabook}. This is exactly the case in Example \ref{burouappu}.

Therefore we have completed the proof of Theorem \ref{mukai1}.

%\medskip

%\medskip

\noindent K.\ Fujita

Research Institute for Mathematical Sciences (RIMS),
Kyoto University, Oiwake-cho, 

Kitashirakawa, Sakyo-ku, Kyoto 606-8502, Japan 

fujita@kurims.kyoto-u.ac.jp
\end{document}